\documentclass[12pt,english]{article}
\usepackage[english]{babel}
\usepackage{amsthm}
\usepackage{amsmath}
\newtheorem{proposition}{Proposition}
\newtheorem{definition}{Definition}
\newtheorem{lemma}{Lemma}
\newtheorem{remark}{Remark}
\usepackage{booktabs}
\usepackage{tabularx}
\usepackage{graphicx}
\usepackage{xcolor}
\usepackage[colorlinks=true, linkcolor=blue, citecolor=blue, urlcolor=blue]{hyperref}
\usepackage{mathptmx}
\usepackage{geometry}
\usepackage{xcolor}
\usepackage{array}
\usepackage{bbding}
\usepackage{multirow}
\usepackage{amsmath}
\usepackage{amssymb}
\usepackage{amsthm}
\usepackage{graphicx}
\usepackage{natbib}
\usepackage{lscape}
\usepackage{hyperref}
\usepackage{comment}
\usepackage{bm}
\usepackage[title]{appendix}
\usepackage{longtable}
\usepackage{mathtools}
\usepackage{chngcntr}
\usepackage{authblk}
\usepackage{babel}
\usepackage{dsfont}
\usepackage{subcaption}
\usepackage{fancyhdr}
\usepackage{cleveref}
\usepackage{abstract}
\makeatletter
\allowdisplaybreaks
\date{}

\fancypagestyle{plain}{%
	\fancyhf{} % sets both header and footer to nothing
	
	\lhead{\footnotesize \includegraphics[height=0.3cm]{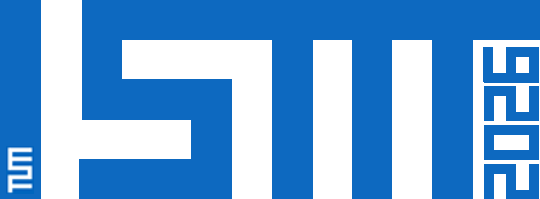}\ \ \textcolor{gray}{\textbf{}}}
	\chead{\footnotesize \textcolor{gray}{\textit{\shortauthors -- \runningtitle}}}
	\rhead{\footnotesize \textcolor{gray}{\thepage}}
}
\makeatletter
\def\blfootnote{\xdef\@thefnmark{}\@footnotetext}
\makeatother

\makeatletter
\def\titlepageext{
	\begin{center}	
	{\parindent0pt
		\rule{0.9\textwidth}{1pt}
		\begin{minipage}[t]{0.25\textwidth}
			\small {\it Keywords:}\\
			\keyword
		\end{minipage}%
		\hspace{3mm}
		\begin{minipage}[t]{0.6\textwidth}
			\small \abstract
		\end{minipage}%
		
		\rule{0.9\textwidth}{2pt}
	}% restore indentation
	\end{center}

	\blfootnote{* Corresponding author. E-mail address: \href{mailto:\corresemail}{\corresemail}.}
}
\makeatother

\usepackage[mathlines, pagewise]{lineno}
\usepackage{etoolbox} %% <- for \pretocmd, \apptocmd and \patchcmd

\newcommand*\linenomathpatchAMS[1]{%
	\expandafter\pretocmd\csname #1\endcsname {\linenomathAMS}{}{}%
	\expandafter\pretocmd\csname #1*\endcsname{\linenomathAMS}{}{}%
	\expandafter\apptocmd\csname end#1\endcsname {\endlinenomath}{}{}%
	\expandafter\apptocmd\csname end#1*\endcsname{\endlinenomath}{}{}%
}

\expandafter\ifx\linenomath\linenomathWithnumbers
\let\linenomathAMS\linenomathWithnumbers
\patchcmd\linenomathAMS{\advance\postdisplaypenalty\linenopenalty}{}{}{}
\else
\let\linenomathAMS\linenomathNonumbers
\fi

\linenomathpatchAMS{gather}
\linenomathpatchAMS{multline}
\linenomathpatchAMS{align}
\linenomathpatchAMS{alignat}
\linenomathpatchAMS{flalign}

\nolinenumbers

\title{Unveiling Traffic Wave of Linear Adaptive Cruise Control: A Second-order Macroscopic Traffic Flow Model}

\def\shortauthors{Z. Li et al.}
\def\runningtitle{Unveiling Traffic Wave of Linear Adaptive Cruise Control: A Second-order Traffic Flow Model}

\author[a]{Zihao Li}
\author[b]{Quyuan Lin}
\author[a]{Fan Pu}
\author[c]{Soyoung Ahn}
\author[a]{Yunlong Zhang}
\author[a]{Jiwan Jiang}
\author[a,*]{Yang Zhou}

\affil[a]{Zachry Department of Civil \& Environmental Engineering, Texas A\&M University, College Station, TX 77843 United States}
\affil[b]{School of Mathematical and Statistical Sciences, Clemson University, Clemson, SC 29634 United States}
\affil[c]{Civil and Environmental Engineering, University of Wisconsin-Madison, Madison, WI 53706 United States}

\def\corresemail{yangzhou295@tamu.edu}

\def\abstract{Traffic waves, the spatiotemporal propagation of congestion, are a key feature of traffic flow. As Adaptive Cruise Control (ACC) systems gain widespread adoption and show promise for improving both efficiency and safety, understanding how these waves evolve under ACC becomes increasingly important. Yet most existing analyses rely on steady-state metrics (e.g., equilibrium spacing) and neglect the ACC control-law parameters, such as feedback gains, that fundamentally shape higher-order traffic dynamics.
To overcome this limitation, we embed the ACC control law directly into the momentum equation while retaining mass conservation law. The result is a higher-order macroscopic model whose dynamics are governed by a second-order partial differential equation equivalent to the linear ACC feedback law. Analyzing the flux Jacobian confirms that the system is strictly hyperbolic, thereby preserving anisotropy and ensuring physical consistency. The derivation also shows that traffic wave evolution depends on both the initial state and the ACC control parameters. We analyze wave-propagation characteristics, linear degeneracy, admissible discontinuities, and their connection to ACC string stability, with the corresponding derivations. Numerical experiments confirm that the second-order model yields markedly lower vehicle-pair speed deviations along wave paths than a first-order model subject to the same non-steady disturbances, underscoring both the necessity of a second-order treatment and the soundness of the proposed framework.}

\def\keyword{Traffic Flow Model\\Traffic Wave\\Adaptive Cruise Control\\Automated Vehicles}

\begin{document}
\maketitle
\titlepageext

\section{Introduction}

Traffic wave formation and propagation are fundamental to understanding traffic dynamics and have been extensively studied for human-driven vehicles (HDVs)~\citep{newell1993simplified-I, newell1993simplified-II,zhang2002kinematic, laval2010mechanism}. Traffic wave models rely on a well-defined flow-density relationship (i.e., the fundamental diagram) and flow conservation, expressed through partial differential equations (PDEs). The seminal Lighthill-Whitham-Richards (LWR) model \citep{lighthill1955kinematic,richards1956shock} conceptualizes traffic as a fluid-like system. Based on a first-order flow-conservation PDE, the LWR model captures key traffic phenomena such as propagation of kinematic waves, particularly shock waves, the moving fronts where traffic abruptly transitions between free-flow and congested regimes. Building on the LWR model, Newell's simplified kinematic wave model \citep{newell1993simplified-I, newell1993simplified-II}, also known as Rankine–Hugoniot shocks \citep{leveque1992numerical},  offers a simplified solution based on the triangular fundamental diagram, in which traffic waves in congested conditions are constant independent of traffic states. 

While the LWR and simplified kinematic wave models can reasonably describe traffic waves observed in the real world, their simplicity limits their ability to capture more nuanced, higher-order traffic dynamics.  This limitation led to the development of higher-order models \citep{ross1988traffic, kuhne1984macroscopic, michalopoulos1993continuum}. The most well-known among them is the Payne--Whitham (PW) model \citep{payne1971model, whitham2011linear}, which introduces additional complexity by incorporating a momentum equation. However, \cite{daganzo1995requiem} critiqued higher-order models for predicting physically unrealistic phenomena, such as negative flows and negative speeds under certain conditions. To address these physical inconsistencies, \citet{zhang1998theory} proposed an improved non-equilibrium model, in which waves were proven to be both linearly and nonlinearly stable \citep{zhang1999analyses}. Nevertheless, this model was later found to still exhibit "gas-like" behavior: perturbations in traffic speed or density can propagate faster than the traffic speed. This implies that a vehicle can influence the behavior of vehicles ahead, which contradicts the nature of real driving behavior. Although these effects decay exponentially over time, the model still lacks physical realism \citep{zhang2000structural}. Later, the Aw–Rascle–Zhang (ARZ) model \citep{aw2000resurrection, zhang2002non} resolved these inconsistencies and was shown to better capture real-world traffic behavior, including the formation and dissipation of shock waves.

With the introduction of automated vehicles (AVs), traffic flow studies involving AVs have received increasing attention \citep{yu2021automated}. Adaptive Cruise Control (ACC), a foundational car-following function of AVs, has been shown to mitigate the propagation of traffic disturbances (in the form of traffic wave) in simulation-based and theoretical studies \citep{talebpour2016influence, zhou2019robust, van2006impact}. However, several field experimental studies suggest otherwise, reporting instances of disturbance amplification \citep{gunter2020commercially, stern2018dissipation, shi2021empirical, makridis2021openacc}. \cite{makridis2021openacc} noted that headway settings may influence disturbance propagation along the platoon (i.e., string stability). These findings suggest that ACC design directly affects the dynamic evolution of disturbances \citep{li2022trade}. As a result, the higher-order dynamics observed in AV traffic can differ significantly from those in human-driven traffic, where disturbance amplification is commonly seen \citep{del2001propagation, mahmassani201650th, zhou2019robust, li2025adaptive}.  First-order formulations, however, cannot reveal those controller-driven effects.  Once a fundamental diagram is fixed, the LWR framework collapses every disturbance onto the wave speed \(d(\rho_e v_e)/d\rho_e\). As a result, any waves that differ only in their ACC feedback gains are predicted to propagate identically, despite having different underlying control dynamics. This built-in equilibrium assumption masks the dynamic features that empirical ACC studies \citep{makridis2024platoon, pu5019798optimal} now highlight.  A higher-order model that retains mass conservation while making the control law explicit is thus essential for an accurate description of traffic waves in AV streams. This highlights the need for higher-order macroscopic traffic models that explicitly incorporate the AV car-following control law, in order to reveal how control parameters shape traffic behavior, particularly in terms of higher-order dynamics.

Several notable studies have investigated traffic wave features or higher-order features in ACC traffic rather than resorting to high-order traffic flow model. For instance, \cite{levin2016multiclass} and~\cite{qin2021lighthill} described wave propagation through an AV fundamental diagram given equilibrium spacing. However, traffic waves in these studies were derived based on steady-state traffic properties, omitting key ACC control parameters (e.g., feedback gains) that govern higher-order traffic dynamics.  \cite{jiang2024dynamic} incorporated a linear feedback law when constructing an extended fundamental diagram and linked that law to traffic hysteresis, yet they did not analyse the resulting wave properties in detail. \cite{zhou2024traffic} derived closed-form wave expressions, but only for the special case of single-frequency steady oscillations. Their approach cannot accommodate compound-frequency disturbances, non-stationary oscillations, or free-flow–to-congestion phase transitions.

To fully characterize wave propagation in pure ACC traffic, this paper develops a second-order macroscopic model that embeds a linear ACC feedback law as the momentum equation while retaining the classical mass-conservation law, thereby mapping microscopic car-following behavior to macroscopic wave dynamics in closed form. Unlike existing approaches that rely on a static fundamental diagram or single-frequency assumptions, the proposed framework keeps the controller gains in the continuum description, so wave behaviour emerges directly from the ACC design rather than from equilibrium curves alone. This gain-aware formulation delivers a unified picture in which disturbances propagate along characteristic paths, transition fronts appear as Rankine–Hugoniot shocks, and the shift from free flow to congestion is pinpointed by an ACC engagement criterion. Together, these elements provide a holistic and analytically transparent alternative to current methods, offering better alignment with empirical observations.

The remainder of the paper is organized as follows. Section 2 derives the ACC-embedded traffic flow model, corresponding traffic wave and establishes its hyperbolicity as well as anisotropy. Section 3 further provides the closed form of the traffic wave and analyzes the corresponding properties under traffic oscillations. Section 4 discusses the traffic wave during discontinuity and provides shock conditions, and phase transition derivations. Section 5 presents numerical experiments that validate the analytical results and compare the proposed model with a first-order LWR benchmark. Section 6 concludes with key findings and directions for future research.

\section{ACC-embedded Traffic Flow Model}
\label{section2}
The starting point for macroscopic traffic flow model is the principle of mass conservation, expressed as:
\begin{equation}
\label{mass_conser}
\frac{\partial \rho(x,t)}{\partial t} + \frac{\partial q(x,t)}{\partial x} = 0,
\end{equation}
where, $\rho(x, t)$ is traffic density and $q(x, t)$ is traffic flow rate. The flow $q$ is typically expressed as a function of density, $q = \rho v(\rho)$, where $v(\rho)$ is the speed-density relationship. Under the equilibrium state, $v(\rho)$ can be represented as $v_e(\rho)$. The corresponding fundamental diagram, describing the equilibrium relationship between flow and density, is given by \(q(\rho)=\rho\,v_e(\rho)\).

Substituting the equilibrium fundamental diagram into the conservation equation yields the LWR model:
\begin{equation}\label{LWR}
\frac{\partial \rho}{\partial t} + \frac{\partial (\rho v_e(\rho))}{\partial x} = 0.
\end{equation}

While this model captures key phenomena such as congestion formation and dissipation, it assumes the equilibrium speed $v_e(\rho)$ and does not explicitly consider the non-equilibrium state. When traffic is not in equilibrium, acceleration and deceleration flows follow distinct paths in the $(\rho,v)$ phase plane, and these paths typically form a hysteresis loop \citep{treiterer1974hysteresis}. To address this limitation, higher-order models introduce a momentum equation that allows the traffic state to deviate from equilibrium. For example, the ARZ model complements the conservation law with a momentum equation that governs the evolution of speed:
\begin{equation}
\label{ARZ_v}
v_t + v v_x = -c(\rho) v_x
\end{equation}
Where \(v_t\) denotes the partial derivative of speed \(v\) with respect to time \(t\), i.e., \(\frac{\partial v}{\partial t}\), and \(v_x\) is the partial derivative of \(v\) with respect to location \(x\). The term \(c(\rho)\) is a function of traffic density \(\rho\) (i.e., $c(\rho) = \rho\, v_e'(\rho)$), representing the characteristic speed (or traffic sound speed) at which small disturbances propagate relative to the traffic flow. Eq.~\eqref{ARZ_v} is derived from the car-following model \(\tau(s_n(t))\, a_n(t) = v_{n-1}(t) - v_n(t)\), where \(\tau(\cdot)\) denotes the driver response time, expressed as a function of the spacing \(s_n(t) = x_{n-1}(t) - x_n(t)\). In this expression, \(a_n(t)\) and \(v_n(t)\) represent the acceleration and speed of the \(n\)-th vehicle, respectively. For detailed derivations, please refer to \cite{zhang1998theory, zhang2002non}.

To model pure ACC traffic, it is essential to account for the algorithmic nature of ACC-equipped vehicles, which follow programmed control laws rather than human driving heuristics. Unlike human drivers, ACC systems typically regulate acceleration based on measured spacing and relative velocity, often implemented through linear feedback control~\citep{li2024disturbances,jiang2024generic,jiang2025stochastic}. Therefore, we replace the momentum equation in the ARZ model with a widely used linear feedback control law to reflect the distinct car-following behavior of ACC vehicles. This formulation is supported by calibration results from existing studies \citep{jiang2024generic,li2024disturbances}, which show that many real-world ACC implementations align closely with linear feedback structures. Moreover, the closed-form nature of this law facilitates analytical tractability and enables the theoretical derivation of disturbance-dampening properties, such as local and string stability \citep{zhou2019robust,li2022trade,li2024enhancing, yue2024hybrid}.

Following~\cite{yi2006macroscopic}, ACC-equipped vehicles operate in two distinct modes. When the traffic density satisfies \(\rho < \rho_c\) (i.e., the density is below capacity) and the vehicle is cruising at its set free-flow speed \(v_f\), the controller remains inactive and no acceleration is applied. When the density satisfies \(\rho_c \le \rho < \rho_j\), where \(\rho_j\) denotes the jam density, or the vehicle speed deviates from \(v_f\), the ACC controller activates a linear feedback law that regulates both spacing and relative velocity with respect to the leader. Accordingly, the control law is given by
\begin{equation}
\label{linear_feedback}
a_n(t) \;=\; k_s(\rho,v)\,\bigl[s_n(t) - s^*_n(t)\bigr]
             \;+\; k_v(\rho,v)\,\bigl[v_{n-1}(t) - v_n(t)\bigr],
\end{equation}
Where $s^*_n(t)$ is the desired spacing of the $n$-th vehicle, determined by the constant headway policy \citep{darbha1999intelligent, yi2006macroscopic}, i.e., $s^*_n(t) = \tau v_n(t) + L$. Here, $\tau$ represents the desired constant headway, and $L$ is the standstill distance. 

The terms $k_s(\rho,v)$ and $k_v(\rho,v)$ are the density-and speed-dependent feedback gains for spacing and speed differences, respectively, are defined Eq.~\eqref{eq:rho_v_gains}. This piecewise formulation reflects the operational characteristics of ACC systems: in the free-flow regime (\(0 < \rho \leq \rho_c\)), vehicles operate in cruise control mode, maintaining a constant speed once the free-flow speed is reached. In the congested regime (\(\rho_c < \rho < \rho_j\)) or whenever $v<v_f$, ACC systems actively regulate both the gap and the relative speed with the immediately preceding vehicle.

\begin{equation}
\label{eq:rho_v_gains}
k_s(\rho,v)=
\begin{cases}
0, & 0 < \rho \leq \rho_c \text{ and } v = v_f\\[6pt]
k_s, &  \text{otherwise}
\end{cases}
\quad
k_v(\rho,v)=
\begin{cases}
0, & 0 < \rho \leq \rho_c \text{ and } v = v_f\\[6pt]
k_v, & \text{otherwise}
\end{cases}
\end{equation}

With the ACC control law specified, we then derive the associated momentum equation.

\begin{proposition}\label{prop:momentum_equation}
The momentum equation derived from the linear feedback control law as Eq.~\eqref{linear_feedback} is given by:
\begin{equation}\label{law-v}
v_t
\;+\;
\left(v - k_v(\rho,v)\,s\right) v_x
\;+\;
\left(-s + v\tau + L\right) k_s(\rho,v)
\;=\;
0.
\end{equation}
\end{proposition}

\begin{proof} 
We denote by $v(x,t)$ and $s(x,t)$ the speed function and the spacing function, and let \(v_n(t) = v(x_n(t), t)\) and \(s_n(t) = s(x_n(t), t)\). Both $v(x,t)$ and $s(x,t)$ are assumed to be sufficiently smooth. The control law can be revised as
\begin{equation}
\label{pde_ACC_law}
\frac{d}{dt} v\left(x_n(t),t\right) = k_s(\rho,v) \left[s\left(x_n(t),t\right) - v\left(x_n(t),t\right)  \tau - L \right] + k_v(\rho,v) \frac{d}{dt} \left[s\left(x_n(t),t)\right)\right].
\end{equation}

Then, we have the following equation, where $(x_n(t),t)$ is dropped to improve readability.

\begin{equation}
vv_x + v_t = k_s(\rho,v) \left[s - v\tau - L\right] + k_v(\rho,v) \left(v s_x + s_t\right).
\end{equation}

By incorporating the conservation of vehicles $s_t + v s_x = s v_x$, which comes from the relation \(s = \frac{1}{\rho}\) and the conservation law (Eq.~\eqref{mass_conser}), Proposition~\ref{prop:momentum_equation} is derived. \end{proof}

To further analyze the traffic wave of ACC-embedded traffic flow model, we organize the traffic flow model into vector form as \(\mathbf{U}_t + \mathbf{J}(\mathbf{U}) \mathbf{U}_x = \mathbf{S}(\mathbf{U})\), where \(\mathbf{U}=[\rho,v]^T\) represents the state vector of traffic variables, \(\mathbf{J}(\mathbf{U})\) is the flux Jacobian matrix, and \(\mathbf{S}(\mathbf{U})\) denotes the source term accounting for external influences on the traffic flow. The eigenvalues of the flux Jacobian matrix represent the characteristic waves, providing insights into how small perturbations in traffic density or speed propagate through the traffic stream \citep{newell1993simplified-I,newell2002simplified,zhang2002non}. In this study, we assume \(\rho > 0\) (i.e., non-vacuum) to ensure that \(s = \frac{1}{\rho}\) is well-defined. However, when \(\rho = 0\) (i.e., vacuum), it does not pose a problem, as the traffic void separates two independent PDE problems, ensuring that the solution of one does not influence the solution of the other \citep{leclercq2007lagrangian,wagner1987equivalence}.

By substituting $q = \rho v$ and applying the product rule, the mass conservation law expands to:
\begin{equation}
\label{mass_conser_long}
\rho_t + (\rho v)_x = \rho_t + \rho_x v + \rho v_x = 0.
\end{equation}

Coupling the mass conservation (Eq.~\eqref{mass_conser_long}) and ACC momentum equations (Eq.~\eqref{law-v}), and using the relation $s=\frac1\rho$, the ACC-embedded traffic flow model is written in matrix form:
\begin{equation}
\label{ACC-traffic}
\begin{bmatrix}
\rho \\ v
\end{bmatrix}_t
+
\begin{bmatrix}
v & \rho \\
0 & v -  k_v(\rho,v)\frac1\rho
\end{bmatrix}
\begin{bmatrix}
\rho \\ v
\end{bmatrix}_x
=
\begin{bmatrix}
0 \\ \left(\frac1\rho - v\tau - L\right)k_s(\rho,v)
\end{bmatrix}.
\end{equation}

% The Jacobian matrix of the flux term is:
% \begin{equation}
% \mathbf{J }(\mathbf{U}) = 
% \begin{bmatrix}
% v & \rho \\
% 0 & v -  k_v(\rho,v)\frac1\rho
% \end{bmatrix}.
% \end{equation}
The flux Jacobian matrix and the source term are:
\begin{equation}
\mathbf{J}(\mathbf{U}) = 
\begin{bmatrix}
v & \rho \\
0 & v - k_v(\rho,v)\dfrac{1}{\rho}
\end{bmatrix},
\qquad
\mathbf{S}(\mathbf{U})=
\begin{bmatrix}
0\\
\left(\dfrac{1}{\rho}-\tau v-L\right)k_s(\rho,v)
\end{bmatrix}.
\end{equation}

The eigenvalues of the Jacobian matrix are given by:
\begin{equation}
\label{eigenvalue}
\lambda_1 = v > \lambda_2 = v - k_v(\rho,v)\frac1\rho.
\end{equation}
The corresponding eigenvectors \(r_1\) and \(r_2\) are derived as follows:
\begin{equation}\label{eigenvector}
r_1 =
\begin{bmatrix}
1 \\ 0
\end{bmatrix},
\quad
r_2 =
\begin{bmatrix}
1 \\[4pt] -\dfrac{k_v(\rho,v)}{\rho^2}
\end{bmatrix}.
\end{equation}

In the free-flow regime (\(0 < \rho \leq \rho_c\), \(v = v_f\)), the feedback gains \(k_s\) and \(k_v\) vanish, and the model reduces to the first-order LWR model. In the congested regime, traffic flow is governed by the strictly hyperbolic second-order balance law (Lemma~\ref{le:strict_hyper}), which can better capture the non-equilibrium dynamics. This model is a specific case of the transition model of non-equilibrium traffic flow \citep{colombo20022,blandin2011general,blandin2013phase} and can be written as:

\begin{equation}\label{eq:phase-trans-pde}
\left\{
\begin{aligned}
&\rho_t + v_f\rho_x = 0 
&&\text{in free-flow } (\Omega_f)\\[6pt]
&\left\{
\begin{aligned}
&\rho_t + \rho_x v + \rho v_x = 0, \\
& v_t + \left(v - k_v\,\frac1\rho\right) v_x + \left(-\frac1\rho + v\tau + L\right) k_s = 0 
\end{aligned}
\right.
&&\text{in congestion } (\Omega_c)
\end{aligned}
\right.
\end{equation}

\begin{lemma}
\label{le:strict_hyper}
The proposed ACC-embedded traffic model is strictly hyperbolic in the congested regime ($\Omega_c$).
\end{lemma}
This result follows from the fact that the eigenvalues of the flux Jacobian matrix are real and distinct (Eq.~\eqref{eigenvalue}). In this regime, strict hyperbolicity ensures that traffic waves propagate with well-defined characteristic speeds, allowing the model to capture non-equilibrium dynamics accurately. This property prevents phenomena such as wave overlap, where multiple waves travel at the same speed and interfere, and eliminates nonphysical solutions, such as negative velocities or unbounded growth, that may arise in ill-posed systems.

\begin{lemma}
\label{le:anisotropic}
%When the characteristic wave speed exceeds the traffic speed, 
The proposed ACC-embedded traffic flow model presents anisotropic behavior.
\end{lemma}
In the proposed ACC-embedded traffic flow model (Eq.~\eqref{ACC-traffic}), the characteristic speed is always less than or equal to the vehicle speed, as \(k_v(\rho,v)\,\frac{1}{\rho} \geq 0\). This inequality holds under the joint conditions of the non-vacuum assumption (\(\rho > 0\)) and the local stability criterion of the ACC control law, which ensures that \(k_v(\rho,v) \geq 0\) \citep{zhou2019robust}. Therefore, the ACC-embedded traffic flow model represents an anisotropic, non-equilibrium traffic system. 

\begin{remark}
The feedback gains in Eq.~\eqref{eq:rho_v_gains} are piecewise-defined to reflect the ACC mode switch. As a result, the flux Jacobian \(\mathbf{J}(\mathbf{U})\) is smooth within each regime but is discontinuous across the cruise-to-control boundary (in particular at \(\rho=\rho_c\) for the density-triggered switch). The model should therefore be interpreted as a two-regime (phase-transition) hyperbolic problem, where the PDEs in \(\Omega_f\) and \(\Omega_c\) are coupled through the ACC switching logic. In the phase-transition setting studied in Section~\ref{sec:dis_wave}, the switching interface is determined by the engagement condition \(s=\tau v_f+L\) (equivalently \(\rho=\rho_c\) under \(s=1/\rho\)), and across this interface the conserved mass satisfies the Rankine--Hugoniot jump condition. Away from the interface, wave propagation is governed by the phase-wise characteristic speeds.
\end{remark}

\begin{remark}
The source term \(\mathbf{S}(\mathbf{U})\) in Eq.~\eqref{ACC-traffic} arises from the ACC headway-keeping objective and appears only in the momentum equation. In congestion, \(k_s>0\) makes \(\mathbf{S}(\mathbf{U})\) a local relaxation forcing toward the ACC time-headway manifold \(s=s^{*}(v)=\tau v+L\) (equivalently \(1/\rho=\tau v+L\)): a positive spacing error \(s-s^{*}(v)\) induces acceleration, whereas a negative error induces braking. This affects solution properties by selecting the equilibrium set: unlike a homogeneous conservation law where any constant state is steady, steady states in \(\Omega_c\) must satisfy \(\mathbf{S}(\mathbf{U})=0\), so equilibria lie on \(1/\rho=\tau v+L\). Since \(\mathbf{S}(\mathbf{U})\) does not enter the mass equation, vehicle conservation remains intact. Moreover, the characteristic structure in \(\Omega_c\) is still determined by the homogeneous flux Jacobian: for a given local state \((\rho,v)\), the characteristic speeds remain \(\lambda_1=v\) and \(\lambda_2=v-k_v/\rho\). The source term affects how \((\rho,v)\) evolves (and thus how \(\lambda_1(\rho,v)\) and \(\lambda_2(\rho,v)\) vary in space and time through the evolving state), rather than redefining the characteristic speeds themselves.
\end{remark}

Note that Appendix~\ref{connection_PTM} provides a comparison between the ACC-embedded traffic flow model and the phase transition model~\citep{blandin2011general,blandin2013phase}, and Appendix~\ref{connection_LWR_ARZ} provides a comparison with the LWR+ARZ formulation~\citep{goatin2006aw,benyahia2016entropy}.

Classical phase-transition conservation-law models have established global well-posedness results \cite{colombo2007global} and corresponding numerical schemes, including Godunov-type methods \cite{chalons2008godunov} and central-upwind constructions \cite{chu2025central}. However, our ACC-embedded model differs in that (i) the congested-regime dynamics constitute a hyperbolic balance law with a nonzero source term, and (ii) the flux Jacobian is discontinuous across the cruise-to-control switching boundary. As a result, the existing well-posedness theory and numerical schemes for conservation-law phase-transition models do not directly apply without additional analysis and interface modifications. In addition, developing a complete numerical scheme for the full phase-transition system is nontrivial ~\citep{chalons2008godunov} because the cruise-to-control switch introduces a moving interface and requires a dedicated interface treatment (i.e., a coupling rule that connects the boundary states and fluxes between \(\Omega_f\) and \(\Omega_c\)). As a first step, Appendix~\ref{app:numerical_scheme} summarizes a numerical scheme for \(\Omega_c\), providing an explicit discretization of the ACC-active balance-law dynamics.

\section{Characteristic Wave Analysis}
\label{section3}
Section~\ref{section2} derives a macroscopic (Eulerian) description for homogeneous ACC traffic, with cruise mode \(\Omega_f\) and active-control mode \(\Omega_c\). Section~\ref{section3} rewrites this continuum model in a vehicle-indexed (Lagrangian) form so that the macroscopic wave description can be evaluated directly on trajectories. The bridge is the continuum approximation \(s \approx 1/\rho\) (equivalently \(\rho \approx 1/s\)) that underlies the micro-to-macro mapping: evaluating the macroscopic wave speed \(\lambda(\rho,v)\) along trajectories using \(\rho_i(t)\approx 1/s_i(t)\) yields wave paths expressed directly in terms of vehicle-pair states \((s_i(t),v_i(t))\). This is consistent with the standard Lagrangian interpretation and particle discretization of second-order traffic models \citep{aw2002derivation,moutari2007hybrid}, and it shows that the pairwise wave-path analysis is the macroscopic wave description written in vehicle-indexed form.

\label{char_wave}
\subsection{Vehicle-Pair Wave Dynamics}
In the ACC‐embedded traffic flow model (Eq.~\eqref{ACC-traffic}), the first characteristic speed \(\lambda_1 = v\) represents pure density perturbations propagated at the vehicle speed. Such perturbations simply convect with the flow and do not capture interactions between vehicle pairs. The upstream‐propagating characteristic wave speed \(\lambda_2 = v - k_v(\rho,v)\,\frac{1}{\rho}\) governs spacing and velocity disturbances and therefore describes how the driving behavior of a leader influences its follower. Hence, at the vehicle-pair level, the wave propagating from vehicle \(n{-}1\) to \(n\) is
\begin{equation}
\label{ACC_wave}
w_{n-1\to n}(t)
= v_{n-1}(t)
  - k_v\bigl(\rho_n(t),\,v_n(t)\bigr)\,\bigl[x_{n-1}(t)-x_n(t)\bigr].
\end{equation}

In the free‑flow regime, the model reduces to the first‑order LWR equation with characteristic speed \(\lambda = v_f\). In the congested regime, where \(k_s(\rho,v)\equiv k_s>0\) and \(k_v(\rho,v)\equiv k_v>0\), the ACC‑embedded model remains second‑order with characteristic speeds \(\lambda_1 = v\) and \(\lambda_2 = v - k_v/\rho\). In the following analysis, we focus on the congested regime, since only in that regime does ACC actively shape traffic wave dynamics. 

For a car-following vehicle pair, given the initial state of the ego vehicle \(n\) and the trajectory of the immediately preceding vehicle \(n-1\), we derive the continuous‑time evolution of the state vector \(z_n(t) = [s_n(t) - s^*_n(t), v_{n-1}(t) - v_n(t)]^T\) via the linear feedback control law \citep{li2024disturbances}, given as:

\begin{subequations}
\begin{equation}
z_n(t) = e^{A_d (t - t_0)} z_n(t_0) 
+ \int_{t_0}^t e^{A_d (t - \varphi)} D_d a_{n-1}(\varphi) d\varphi,
\end{equation}
\text{Where:}
\begin{equation}
A_d = 
\begin{bmatrix}
-\tau k_s & 1 - \tau k_v \\
-k_s & -k_v
\end{bmatrix},
\quad
D_d = 
\begin{bmatrix}
0 \\
1
\end{bmatrix}.
\end{equation}
\end{subequations}

Then, the spacing between the vehicle pairs is given by:
\begin{equation}
\begin{aligned}
s_n(t)= x_{n-1}(t) - x_n(t)= & \left(\mathbf{G_0} - \mathbf{G_1} \tau\right) \Bigg( e^{A_d (t - t_0)} z_n(t_0) 
+ \int_{t_0}^t e^{A_d (t - \varphi)} D_d a_{n-1}(\varphi) d\varphi \Bigg) \\
& + \tau \left( v_{n-1}(t_0) + \int_{t_0}^t a_{n-1}(\varphi) d\varphi \right) + L.
\end{aligned}
\label{eq:space_eq}
\end{equation}
% \begin{equation}
% \begin{aligned}
% v_n(t) = v_{n-1}(t_0) + \int_{t_0}^t a_{n-1}(\varphi) \, d\varphi-\mathbf{G_1}e^{A_d (t - t_0)} z_n(t_0) 
% -\mathbf{G_1}\int_{t_0}^t e^{A_d (t - \varphi)} D_d a_{n-1}(\varphi) d\varphi
% \end{aligned}
% \end{equation}

Where \(\mathbf{G}\) is the column selection matrix: \(\mathbf{G_0} =[1,0]\), which extracts the first column, and  \(\mathbf{G_1} =[0,1]\), which extracts the second column.

From Eqs.~\ref{ACC_wave} and~\ref{eq:space_eq}, it is demonstrated that the characteristic wave speed depends on multiple factors, including the state of the immediately preceding vehicle $( v_{n-1}(t), a_{n-1}(t))$, the initial states of the vehicle pair $(z_n(t_0))$, and the control parameters $(\tau,k_s,k_v)$.

\begin{remark}
    Unlike \cite{zhou2024traffic}, which focuses solely on the steady oscillation case of a single frequency, this study  extends the analysis by relaxing the steady-state oscillation assumption to account for initial disturbances. It is noted that both methods are continuous, providing a temporally and spatially descriptive analysis.
\end{remark}

\subsection{Wave Dynamics in Traffic Oscillations}
To further analyze the properties of characteristic waves in ACC traffic flow, the position of the leading vehicle is decomposed into a nominal component and an oscillation component \citep{li2014stop}.
\begin{equation}
\label{eq:x_norm+oss}
\begin{aligned}
x_0(t) = v_e t + \hat{x}_0(t) = v_e t + \sum_{m=1}^{M} A_m \sin(\omega_m t + \phi_m)
\end{aligned}
\end{equation}
Where $v_e t$ is the nominal component, determined by the equilibrium speed, and $\hat{x}_0(t)$ is the oscillation component, decomposed into a sum of sinusoidal waves using a Fourier transform \citep{jiang2024dynamic}. $m$ is the index of the oscillatory waves, and $A_m$ is the amplitude of an oscillatory wave with frequency $\omega_m$ and phase shift $\phi_m$. We can also derive the speed of the leading vehicle as:
\begin{equation}
\label{eq:v_norm+oss}
\begin{aligned}
v_0(t) = v_e + \hat{v}_0(t) = v_e + \sum_{m=1}^{M} A_m \omega_m \cos(\omega_m t + \phi_m)
\end{aligned}
\end{equation}
Based on Eqs.~\ref{eq:x_norm+oss} and~\ref{eq:v_norm+oss}, the position and speed of the following vehicle \(n\) are obtained using the transfer function \(G(j\omega_m)\) derived from the linear feedback control in Eq.~\eqref{linear_feedback} through the Laplace transform.
\begin{subequations}
\label{eq:vehicle_motion}
\begin{align}
x_n(t) &= (v_e t - n x_e) + \Delta \hat{x}_n(t) \nonumber\\
       &= (v_e t - n x_e) + \sum_{m=1}^{M} A_m \bigl|G(j\omega_m)\bigr|^n
           \sin\!\bigl(\omega_m t + \phi_m + n\,\angle G(j\omega_m)\bigr),\\
v_n(t) &= v_e + \hat{v}_n(t) \nonumber\\
       &= v_e + \sum_{m=1}^{M} A_m \bigl|G(j\omega_m)\bigr|^n \omega_m
           \cos\!\bigl(\omega_m t + \phi_m + n\,\angle G(j\omega_m)\bigr).
\end{align}
\end{subequations}

\begin{subequations}
\noindent \text{Where:}
\begin{equation}
G(s) = \frac{X_n(s)}{X_{n-1}(s)} = \frac{k_s + k_v s}{s^2 + k_s \tau^* s + k_v s},
\end{equation}
\text{Given $s=jw$, the corresponding norm and phase shift are:}
\begin{equation}
|G(j\omega)| = \frac{\sqrt{
\left(-k_s + k_v (k_s \tau^* + k_v)\right)^2 \cdot \omega^2 
+ \left[k_s (k_s \tau^* + k_v) + k_v \omega^2\right]^2
}}{
\omega \cdot \left[(k_s \tau^* + k_v)^2 + \omega^2\right]},
\end{equation}
\begin{equation}
\angle G(j\omega) = \tan^{-1}\left(
\frac{k_s (k_s \tau^* + k_v) + k_v \omega^2}{\omega \left[k_s - k_v (k_s \tau^* + k_v)\right]}
\right).
\end{equation}
\end{subequations}

\begin{proposition}Under the steady compounding oscillation case, the wave speed consists of two components: a nominal part and an oscillatory part. The oscillatory part is a summation of each single oscillation case. If the frequencies are rationally related, the periods of the wave speed oscillations are the least common multiple of \( T_m = \frac{2\pi}{\omega_m} \) for \(1 \leq m \leq M\).
\end{proposition}

\begin{proof}
Under the steady compounding oscillation case (Eq.~\eqref{eq:vehicle_motion}), we have:
\begin{subequations}
\begin{equation}
x_{n-1}(t) - x_n(t) = x_e + \sum_{m=1}^{M} A_m \left(|G(j\omega_m)|^{n-1} - |G(j\omega_m)|^n\right) 
\sin\left(\omega_m t + \phi_m + (n-1) \angle G(j\omega_m)\right)
\end{equation}
\begin{equation}
v_{n-1}(t) = v_e + \sum_{m=1}^{M} A_m |G(j\omega_m)|^{n-1} \omega_m 
\cos\left(\omega_m t + \phi_m + (n-1) \angle G(j\omega_m)\right)
\end{equation}
\end{subequations}

The wave speed derived from the ACC traffic flow model is given in Eq~\ref{ACC_wave}, and it follows
\begin{equation}
\begin{aligned}
W_{n-1\to n}(t) = & 
\left[v_e + \sum_{m=1}^{M} A_m |G(j\omega_m)|^{n-1} \omega_m 
\cos\left(\omega_m t + \phi_m + (n-1) \angle G(j\omega_m)\right)\right] \\
& - k_v \bigg[x_e + \sum_{m=1}^{M} A_m \left(|G(j\omega_m)|^{n-1} - |G(j\omega_m)|^n\right) 
\sin\left(\omega_m t + \phi_m + (n-1) \angle G(j\omega_m)\right)\bigg]
\end{aligned}
\end{equation}
By simplifying the above equations, we have:
\begin{subequations}
\begin{align}
W_{n-1\to n}(t) &= v_e - k_v x_e + \sum_{m=1}^{M} R_{n,m} \cos\big(\Theta_m(t) + \phi_{n,m}\big) 
\end{align}
\noindent Where:
\begin{align}
R_{n,m} &= \sqrt{\big(A_m |G(j\omega_m)|^{n-1} \omega_m\big)^2 + \big(k_v A_m \big(|G(j\omega_m)|^{n-1} - |G(j\omega_m)|^n\big)\big)^2} \\
\phi_{n,m} &= \tan^{-1}\left(\frac{-k_v \big(|G(j\omega_m)|^{n-1} - |G(j\omega_m)|^n\big)}{A_m |G(j\omega_m)|^{n-1} \omega_m}\right) \\
\Theta_m(t) &= \omega_m t + \phi_m + (n-1) \angle G(j\omega_m)
\end{align}
\end{subequations}

The periods of the wave speed oscillations are the least common multiple of  \( T_m = \frac{2\pi}{\omega_m} \) for \(1 \leq m \leq M\), provided the frequencies are rationally related. If the frequencies are not rationally related, meaning the ratio of any two frequencies \(\omega_i\) and \(\omega_j\) cannot be expressed as a fraction of two integers (i.e., \(\frac{\omega_i}{\omega_j}\) is irrational), the system exhibits quasi-periodicity rather than a single well-defined period. In a quasi-periodic system, the oscillations are composed of multiple frequencies that combine in a way that the pattern never exactly repeats over time.
\end{proof}

\begin{lemma}\label{lem3}
Under the single steady oscillation case, the wave speed is expressed as 
\begin{equation}
W_{n-1\to n}(t) = v_e - k_v x_e + R_n \cos(\Theta(t) + \phi_n). 
\end{equation}
\end{lemma}

This derivation aligns with the findings of~\cite{zhou2024traffic}, capturing the intra- and inter heterogeneity properties. Intra-heterogeneity refers to the time-varying nature of the wave speed for a single leader–follower pair, whereas inter-heterogeneity describes how that wave pattern differs from one vehicle pair to the next along the platoon. In this oscillatory scenario, the traffic wave also exhibits oscillations in the time domain and demonstrates higher-order dynamic features that distinguish it from traditional methods.

String stability in the ACC-embedded traffic flow model governs the propagation of oscillatory disturbances through a platoon of vehicles, determined by the norm of the transfer function \( |G(j\omega)| \)~\citep{zhou2019robust,li2024enhancing}. This section analyzes the impact of \( |G(j\omega)| < 1 \), \( |G(j\omega)| = 1 \), and \( |G(j\omega)| > 1 \) on the characteristic wave speed \( W_{n-1\to n}(t) \) in the congested regime.

\paragraph{String Stable Case~\( |G(j\omega)| < 1 \), $\forall \omega>0$:}
When \( |G(j\omega_m)| < 1 \), the oscillatory term \( A_m |G(j\omega_m)|^n \) decreases exponentially with \( n \), reducing the amplitude \( R_{n,m} \). As \(R_{n,m}\to 0\), the flow becomes spatially uniform and no longer
supports a discernible travel disturbance; in other words, the
string-stable ACC law has completely attenuated the initial
perturbation, leaving a steady stream governed solely by the equilibrium
traffic state \((x_e,v_e)\).

\paragraph{Marginally Stable Case~\( |G(j\omega)| = 1 \), $\forall \omega>0$:}
When \( |G(j\omega_m)| = 1 \), the ACC law becomes dynamically equivalent to Newell’s car-following model; the oscillation amplitude in Eq.~\eqref{eq:vehicle_motion} therefore remains unchanged with vehicle index $n$. The resulting wave speed reduces to
\begin{equation}
W_{n-1\to n}(t) = v_e - k_v x_e + \sum_{m=1}^{M} A_m \omega_m \cos\big(\omega_m t + \phi_m + (n-1) \angle G(j\omega_m)\big).
\end{equation}
The wave speed exhibits persistent oscillations with constant amplitude across the platoon, resulting in periodic traffic waves that neither grow nor decay. This leads to sustained oscillatory wave patterns, which may reduce traffic flow efficiency due to ongoing fluctuations.

\paragraph{String Unstable Case~\( |G(j\omega)| > 1 \):}
When \( |G(j\omega_m)| > 1 \), the term \( |G(j\omega_m)|^n \) grows exponentially, amplifying \( R_{n,m} \). This causes large fluctuations in \( W_{n-1\to n}(t) \), with increasing oscillatory amplitudes upstream.

\section{Wave Propagation with Discontinuities}
\label{sec:dis_wave}
In the previous section, we analyzed characteristic‐wave propagation in a vehicle pair under the assumption that the traffic density \(\rho(x,t)\) and speed \(v(x,t)\) evolve smoothly, which allows us to track how disturbances travel along characteristic directions. However, hyperbolic PDEs can develop discontinuities even from a smooth initial traffic state, and the propagation of these discontinuities must be treated separately. In this section, we examine the formation and propagation of discontinuities within the proposed ACC‐embedded traffic flow model.

\subsection{Linear Degeneracy of Characteristic Families}
We begin by analyzing linear degeneracy because it precisely distinguishes between characteristic fields that admit only contact discontinuities and those that can steepen into shocks or rarefactions.

\begin{definition}
A characteristic field \(i\) is linearly degenerate if its associated eigenvalue \(\lambda_i(\mathbf{U})\) satisfies
\[
  \nabla \lambda_i(\mathbf{U}) \cdot r_i(\mathbf{U}) = 0
  \quad \text{for all admissible states } \mathbf{U},
\]
where \(r_i(\mathbf{U})\) is the corresponding eigenvector. Otherwise, the characteristic field is called genuinely nonlinear.
\end{definition}

\begin{proposition}
\label{linear_degenrate}
Linear degeneracy holds in both the free-flow and congested regimes. When the vehicle state transitions from free flow to the congested regime (i.e., when spacing becomes small enough to activate ACC), \(k_v\) changes discontinuously, rendering the field genuinely nonlinear and producing shock or rarefaction waves.
\end{proposition}

Recall the eigenvalues and eigenvectors in Eqs.~\ref{eigenvalue}–\ref{eigenvector}.  
For the first characteristic field,
\begin{equation}\label{eq:linear-degen-1}
\nabla\lambda_1 \cdot r_1
  = \begin{bmatrix}0 & 1\end{bmatrix}
    \begin{bmatrix}1 \\ 0\end{bmatrix}
  = 0.
\end{equation}
Hence $\lambda_1$ is linearly degenerate: any disturbance aligned with this field is transported as a pure contact wave at the vehicle speed, without ever steepening into a shock or spreading into a rarefaction.

For the second characteristic field, we have
\begin{equation}
\nabla\lambda_2\!\cdot\!r_2 \;=\begin{bmatrix}
-\dfrac{1}{\rho}\dfrac{\partial k_v}{\partial \rho} + \dfrac{k_v}{\rho^{2}},
&
1 - \dfrac{1}{\rho}\dfrac{\partial k_v}{\partial v}
\end{bmatrix}
\begin{bmatrix}
1 \\[4pt]
-\dfrac{k_v}{\rho^{2}}
\end{bmatrix}
\,=\,
-\dfrac{1}{\rho}\dfrac{\partial k_v}{\partial \rho}
+ \dfrac{k_v}{\rho^{3}}\dfrac{\partial k_v}{\partial v}.
\label{eq:linear-degen}
\end{equation}
% In the ACC‑embedded model, the first characteristic field is always linearly degenerate. In the second characteristic field, Because \(k_v\) remains constant in the free flow and congested regimes respectively, the degeneracy condition holds in each regime. In particular, when the vehicle state transitions from free flow to the congested regime (i.e.\ when spacing becomes small enough to activate ACC), \(k_v\) changes discontinuously, rendering the field genuinely nonlinear and producing shock or rarefaction waves. 

For the second field, linear degeneracy holds within each regime as \(k_v\) remains constant. However, at the transition from free flow to congestion, the discontinuous change in \(k_v\) introduces genuine nonlinearity.

\begin{remark}
\label{remark:congestion_wave}
Based on Proposition~\ref{linear_degenrate}, if the vehicle state remains in the congested regime, smooth disturbances do not generate shock waves (i.e., discontinuity). In this case, the vehicle-pair wave speed is given by Eq.~\eqref{ACC_wave}, which can better describe how the wave propagates along the vehicle string.
\end{remark}

\subsection{Phase Transition}

The principal difficulty in analysing phase transitions is to identify precisely when traffic leaves free flow, how the ensuing transition front propagates, and which wave description—shock or characteristic—applies on either side of that front.  A physically apparent marker of this boundary is the instant at which each ACC-equipped vehicle switches from cruise mode to active gap regulation. In a time–space diagram, these \emph{engagement points} (red points) form a broken curve (red curve) that separates the free-flow state \((v_f,1/s_l)\) from the critical state \((v_f,1/s_c)\); see Figure~\ref{fig:wave+traj}~(right). The slope of each segment of this curve should be different, as it represents the speed of the \emph{engagement wave} (Remark~\ref{engage_diff}).

\paragraph{Engagement time.}
For vehicle \(i\), the ACC controller engages at the first instant when the actual spacing  \(s_{i-1\rightarrow i}(t)=x_{i-1}(t)-x_i(t)\) equals the critical value \(s_c=\tau v_f + L\). Accordingly, the engagement time \(t_i^{\star}\) is defined as
\begin{equation}
\label{eq:engage_condition}
% s_{i-1\rightarrow i}(t_i^{\star})=s_c,
% \qquad
t_i^{\star}=\inf\{\,t>0: s_{i-1\rightarrow i}(t)=s_c\}.
\end{equation}
Let the follower \(i\) travel in free flow regime with position \(x_i(t)=x_i(0)+v_f t\). Assume the leader \(i-1\) has acceleration \(a_{i-1}(t)\) and the same initial speed \(v_f\); its position is therefore
\begin{equation}
x_{i-1}(t) = x_{i-1}(0) + v_f t + \int_{0}^{t} (t - \sigma)\, a_{i-1}(\sigma)\, d\sigma,
\end{equation}
The engagement time satisfies the implicit equation
\begin{equation}
\label{eq:engage_implicit}
s_i(0)
+\int_{0}^{t_i^{\star}}\!(t_i^{\star}-\sigma)\,a_{i-1}(\sigma)\,d\sigma
-s_c=0,
\end{equation}
which has a unique root for any admissible acceleration profile \(a_{i-1}(t)\).  For piece-wise constant accelerations the root is analytical. In general it is obtained numerically.

\begin{lemma}[Engagement wavefront]
\label{prop:engage_wave}
Let \(\bigl(t_i^{\star},x_i(t_i^{\star})\bigr)\) and \(\bigl(t_{i-1}^{\star},x_{i-1}(t_{i-1}^{\star})\bigr)\) be the ACC engagement points of two successive vehicles as determined by Eq.~\eqref{eq:engage_condition}–\eqref{eq:engage_implicit}.  The wave connecting these points propagates with speed
\begin{equation}
\label{eq:engage_speed}
c^{\star}_{i-1\rightarrow i}=
\frac{x_{i-1}(t_{i-1}^{\star})-x_i(t_i^{\star})}
     {t_{i-1}^{\star}-t_{i}^{\star}}\;.
\end{equation}

Connecting all such segments across the platoon traces the phase transition front from free flow to congestion.
\end{lemma}

% \begin{proof}
% Substituting the kinematic expressions for \(x_i(t)\) and \(x_{i-1}(t)\) into the spacing definition and imposing \(s_{i-1\rightarrow i}(t_i^{\star})=s_c\) produces Eq.~\eqref{eq:engage_implicit}, whose unique solution is \(t_i^{\star}\); repeating the procedure for vehicle \(i-1\) yields \(t_{i-1}^{\star}\).  Inserting the two engagement points into Eq.~\eqref{eq:engage_speed} gives the segment speed, and chaining all
% segments furnishes the complete wavefront, thereby proving the proposition.
% \end{proof}

\begin{remark}
\label{engage_diff}
Because \(t_i^{\star}\) depends on the initial gap and the leader’s acceleration profile via Eq.~\eqref{eq:engage_implicit}, these engagement times vary from vehicle to vehicle. Consequently \(c^{\star}_{i-1\rightarrow i}\neq c^{\star}_{i\rightarrow i+1}\) in general, and the engagement wave is not uniform along the platoon.  
\end{remark}

When the transition ends, ACCs fully enter the congested regime. We define this moment as when the vehicle’s spacing reaches \(s_e\). Here, \(s_e\) is determined by the constant time-headway policy using the equilibrium speed \(v_e\), where \(v_e\) is the equilibrium component extracted from the leading vehicle’s movement in the oscillation decomposition (see Eq.~\eqref{eq:v_norm+oss} in the steady oscillation case). At that instant, the vehicle is considered to have fully entered the congested regime, and subsequent disturbances propagate along that regime’s characteristic waves. In Figure~\ref{fig:wave+traj}~(left), the vehicle state gradually moves from the red marker to the purple marker and then remains within its equilibrium deviation envelope (i.e., hysteresis).

\begin{figure}[!h]\centering
    \includegraphics[width=0.99\textwidth]{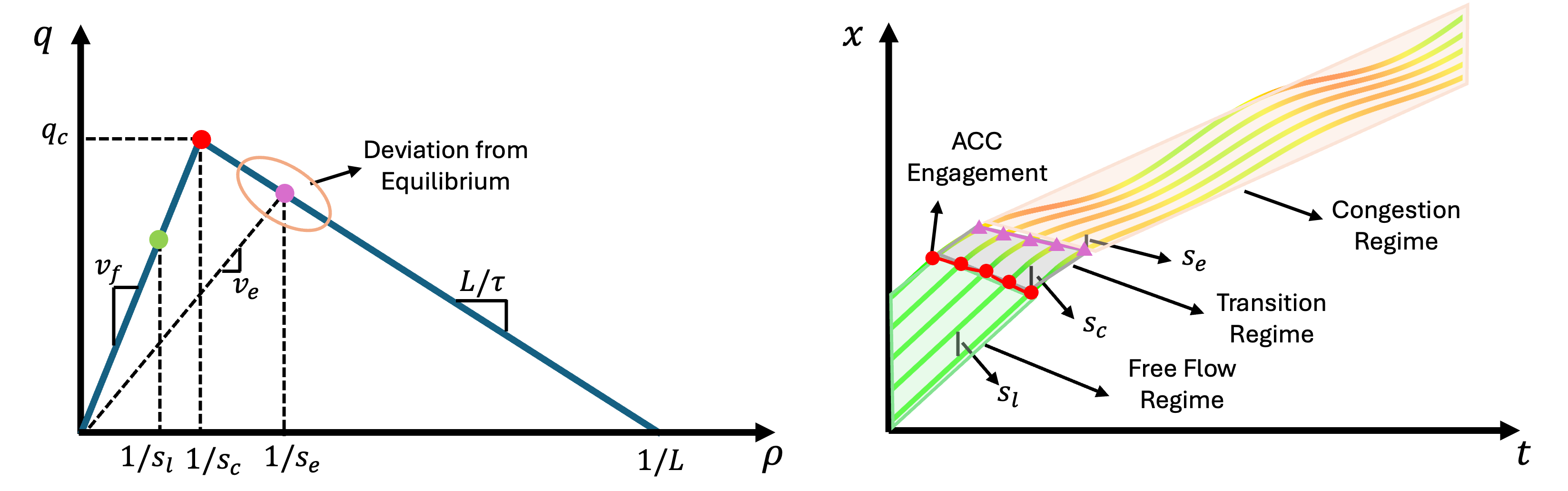}
    \vspace{-10pt} % Adjust vertical space between figures and caption
    \caption{Illustration of wave under the phase transition: equilibrium fundamental diagram (left) and vehicle trajectories (right)}
    \label{fig:wave+traj}
\end{figure}

\begin{remark}
\label{shock}
During the cruise-to-control transition, the traffic state moves from the critical state
$(v_f,1/s_c)$ to another steady state $(v_e,1/s_e)$ within the congested regime $\Omega_c$.
Across the transition front, the solution can be treated as a moving discontinuity whose
propagation speed is governed by the Rankine--Hugoniot condition for the conserved mass
equation~\citep{leveque1992numerical}. In particular, the shock speed between the left and
right states $(\rho_l,v_l)$ and $(\rho_r,v_r)$ satisfies
\begin{equation}\label{eq:shock_speed}
c
= \frac{\rho_r v_r - \rho_l v_l}{\rho_r - \rho_l}.
\end{equation}
Using the continuum correspondence $\rho = 1/s$ and the steady-state relation in $\Omega_c$,
this reduces to $c = -L/\tau$ for the transition considered here.
\end{remark}

% \begin{remark}
% \label{shock}
% During the transition, when traffic transits from the critical state $(v_f, 1/s_c)$ to another steady state $(v_e, 1/s_e)$ gradually over the congested regime, the wave propagates as a discontinuity governed by the Rankine--Hugoniot condition~\citep{leveque1992numerical}. The shock speed between vehicles within the congested regime for vehicle pair \(i-1\) and \(i\) always follows
% \begin{equation}\label{eq:shock_speed}
% c
% = \frac{\rho_r v_r - \rho_l v_l}{\rho_r - \rho_l}
% = -\frac{L}{\tau},
% \end{equation}
% \end{remark}

Therefore, in the phase transition case, we use ACC engagement events to determine the transition front (Lemma~\ref{prop:engage_wave}). During the phase transition, disturbances propagate as a shock wave (Remark~\ref{shock}) and once each vehicle’s spacing reaches \(s_e\), indicated by the purple triangle points in Figure~\ref{fig:wave+traj}~(right), the vehicle has fully entered the congested regime and disturbances propagate along the characteristic waves (Remark~\ref{remark:congestion_wave}).

\section{Numerical Experiments}

In this section, we conduct several numerical experiments to comprehensively validate the derived traffic wave under the following cases: steady single oscillation (case 1), steady single oscillation with an initial disturbance (i.e., a cut-in maneuver, case 2), steady compound oscillation (case 3), and phase transition from the free-flow to the congested regime (case 4). We then demonstrate the empirical applicability of the proposed traffic wave by using the ACC trajectory from the OpenACC dataset~\citep{makridis2021openacc} and simulating the follower vehicles with a calibrated linear feedback controller~\citep{jiang2024dynamic, jiang2025stochastic}.

% The traffic wave describes how disturbances propagate through a vehicle platoon, with vehicle speeds along the wave path remaining consistent and exhibiting low variation to clearly characterize the propagation \citep{laval2011hysteresis,ahn2013method,taylor2015method,pu5019798optimal}. 

The traffic wave describes how disturbances propagate through a vehicle platoon. A meaningful propagation path should track the same disturbance feature across vehicles, so the speeds sampled along the path should be relatively consistent and exhibit low variation~\citep{laval2011hysteresis,ahn2013method,taylor2015method,pu5019798optimal}. An incorrect path may connect different phases of the oscillation (or different parts of the disturbance), producing larger speed variability. Hence, smaller speed variability (or smaller successive speed differences) along a path indicates a more coherent match to the observed disturbance evolution on trajectories. We introduce a metric grounded in the speed deviations measured along each wave path, as illustrated in Figure~\ref{fig:illus_metric}. Consider the set $\mathbf{P}$ of all wave paths, where each path originates at a point of the leading vehicle's trajectory and connects to the trajectory points of the following vehicles. For any path $p \in \mathbf{P}$, we collect the speeds at its sampled positions into the sequence $\mathbf{V}_p = \{v_{p,1}, v_{p,2}, \dots, v_{p,n_p}\}$, with $n_p$ indicating the total number of samples on path $p$. Based on these sequences, we build the set $\mathbf{Y}$ comprising every speed deviation between successive points across all paths, defined as

\begin{equation}\label{metric}
    \mathbf{Y} = \bigcup_{p\in \mathbf{P}}\left\{v_{p,i+1}-v_{p,i} : i=1,\dots,n_p-1\right\},
\end{equation}
whose distribution shows a lower mean, lower median, and a narrower interquartile range, indicating less speed variation along the wave path. This suggests that the identified path is more consistent with the disturbance propagation on trajectories, as vehicle-pair speed deviations remain small \citep{ahn2013method,pu5019798optimal}. 

In addition, to provide a direct macroscopic--microscopic validation, we numerically solve the congested-regime balance law in $\Omega_c$ using the finite-volume scheme summarized in Appendix~\ref{app:numerical_scheme}.
We then compare the PDE solution with Eulerian fields mapped from the microscopic ACC simulation and quantify the discrepancy using root-mean-square errors (RMSE) for $v$ and $\rho$.
The full procedure and results are reported in Appendix~\ref{app:pde_micro_validation}.

In our numerical experiments, we include a constant-speed wave baseline following the first-order (kinematic-wave) assumption used in \cite{laval2011hysteresis}. This baseline is used only as a reference for comparison. In contrast, the proposed second-order macroscopic model implies state-dependent, time-varying characteristic speeds, which can lead to time-varying wave propagation on trajectories. The default parameters for the numerical experiments are provided in Table~\ref{default_para}.

\begin{figure}[!h]
  \centering
  \setlength{\abovecaptionskip}{2pt}
  \setlength{\belowcaptionskip}{-6pt}
  \includegraphics[width=0.7\linewidth]{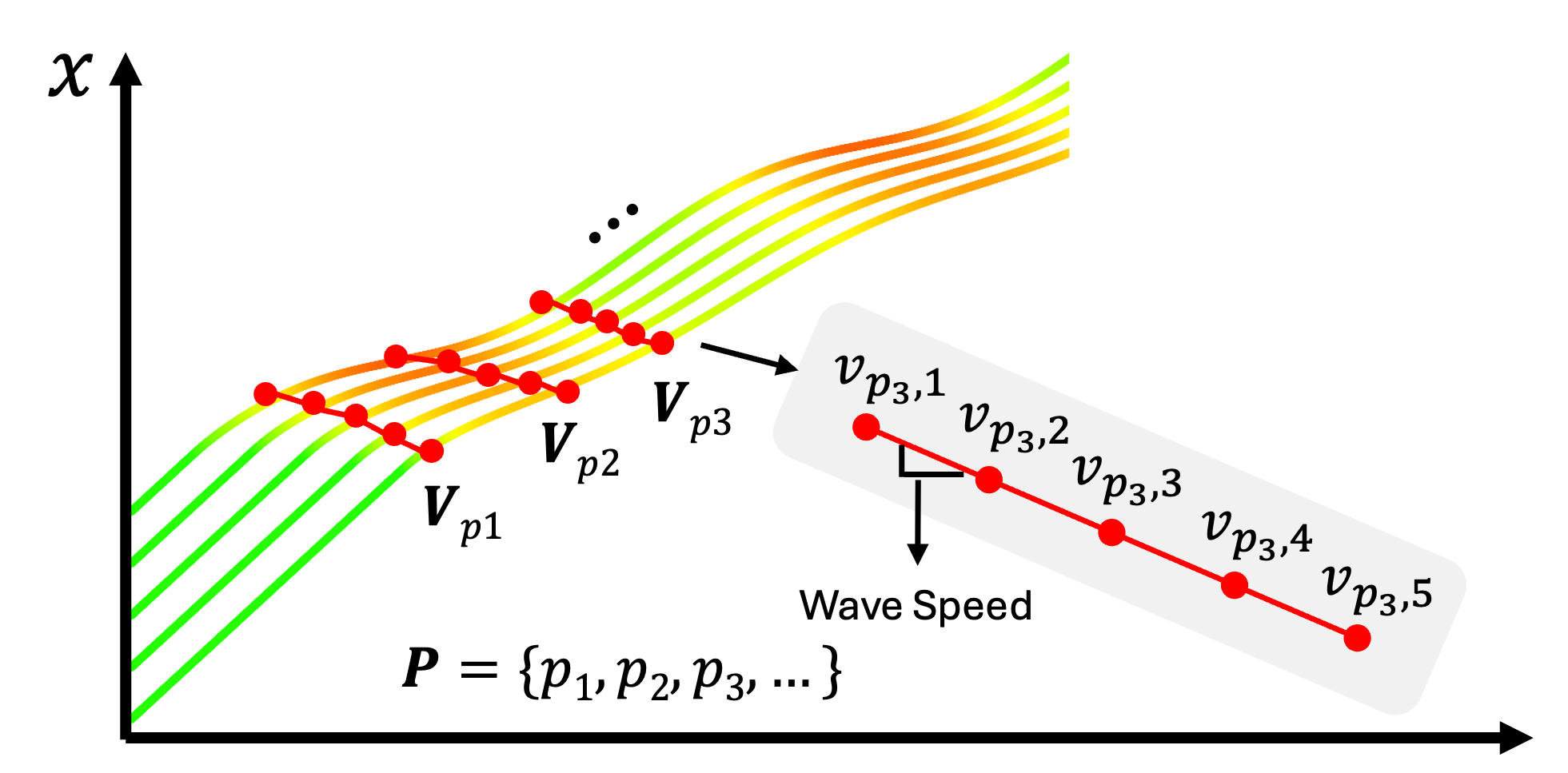}
  \caption{Illustration of the vehicle-pair speed deviation.}
  \label{fig:illus_metric}
\end{figure}

\begin{table}[!h]
\centering
\caption{Default parameter settings for numerical experiments.}
\vspace{-5pt}
    \begin{tabularx}{1\textwidth}{lXlX}
\toprule
\multicolumn{2}{l}{\textbf{Parameters}}                        & \multicolumn{2}{l}{\qquad\textbf{Values}}                          \\ \midrule
\multicolumn{2}{l}{Vehicle number in the platoon $N$} & \multicolumn{2}{l}{\qquad$4$}                             \\
\multicolumn{2}{l}{Equilibrium speed $v_e$}           & \multicolumn{2}{l}{\qquad$10~(\text{m/s})$}               \\
\multicolumn{2}{l}{Desired time gap $\tau$}           & \multicolumn{2}{l}{\qquad$1.2~(\text{s})$}                \\
\multicolumn{2}{l}{Standstill distance $L$}           & \multicolumn{2}{l}{\qquad$5~(\text{m})$}                  \\
\multicolumn{2}{l}{Spacing deviation gain $k_s$}      & \multicolumn{2}{l}{\qquad$0.8~(\text{s}^{-2})$}             \\
\multicolumn{2}{l}{Speed difference gain $k_v$}       & \multicolumn{2}{l}{\qquad$1.4~(\text{s}^{-1})$}             \\ \midrule
\multicolumn{2}{l}{\textbf{Single oscillation case}}  & \multicolumn{2}{l}{\textbf{Compound oscillation case}} \\ \midrule
Amplitude $A$               & \qquad $20~(\text{m})$                & $A_1,A_2$           & \qquad $20~(\text{m}),~10~(\text{m})$      \\
Angular frequency $\omega$          & \qquad $0.16\pi~(\text{rad/s})$          & $\omega_1,\omega_2$ & \qquad $0.16\pi~(\text{rad/s}),~0.32\pi~(\text{rad/s})$                    \\
Phase shift $\phi$          & \qquad $0~(\text{rad})$               & $\phi_1,\phi_2$     & \qquad $0~(\text{rad}),~0.5\pi~(\text{rad})$ \\ \bottomrule
\label{default_para}
\vspace{-10pt}
\end{tabularx}
\end{table}

\newpage
\subsection{Simulation based Validation}

The overall results are summarized in Table~\ref{tab:perform_compare}, which lists absolute vehicle-pair speed deviations and reveals key statistical differences among the cases. For each scenario we also plot the probability density function of speed deviations, that is, the distribution of \(\mathbf{Y}\) defined in Eq.~\eqref{metric}. The proposed traffic wave adapts to changing conditions and consistently outperforms the constant speed wave, with the advantage most pronounced under strong oscillations and cut-in disturbances (Cases~2 and~3). During the phase transition in Case~4, when ACC engagement shifts the traffic state from free flow to congestion, the model likewise captures wave propagation more accurately and yields the smallest vehicle-pair speed deviations.

\begin{table}[ht]
	\centering
	\caption{Performance comparison under different cases}
        \vspace{-5pt}
	\resizebox{\textwidth}{!}{%
	\begin{tabular}{lcccccccc}
		\hline
		               & \multicolumn{2}{c}{\textbf{Case 1}} & \multicolumn{2}{c}{\textbf{Case 2}} & \multicolumn{2}{c}{\textbf{Case 3}} & \multicolumn{2}{c}{\textbf{Case 4}} \\
		\hline
		               & Proposed                            & First-order                         & Proposed                            & First-order                        & Proposed      & First-order & Proposed      & First-order \\
		\hline
		Mean           & \textbf{1.02}                       & 1.21                                & \textbf{0.80}                       & 1.14                               & \textbf{1.27} & 2.01        & \textbf{0.89} & 1.29        \\
		Median         & \textbf{0.99}                       & 1.26                                & \textbf{0.72}                       & 1.11                               & \textbf{1.76} & 1.87        & \textbf{0.85} & 1.34        \\
		Lower quartile & \textbf{0.54}                       & 0.69                                & \textbf{0.40}                       & 0.63                               & \textbf{0.56} & 0.92        & \textbf{0.48} & 0.77        \\
		Upper quartile & \textbf{1.44}                       & 1.71                                & \textbf{1.12}                       & 1.69                               & \textbf{1.76} & 2.83        & \textbf{1.29} & 1.77        \\
		Max            & \textbf{2.33}                       & 2.56                                & \textbf{1.97}                       & 2.87                               & \textbf{3.47} & 5.81        & \textbf{2.54} & 2.79        \\
		Min            & 0.00                                & 0.00                                & \textbf{0.00}                       & 0.03                               & 0.03          & 0.03        & \textbf{0.00} & 0.07        \\
		\hline
	\end{tabular}
	} \label{tab:perform_compare}
    \vspace{-10pt}
\end{table}

\subsubsection{Case 1: Single Steady Oscillation}
In the single steady oscillation case, the proposed traffic wave yields slightly smaller vehicle-pair speed deviations than the constant-speed benchmark. Since all follower vehicles start in the equilibrium car-following state and remain in the congested regime, we apply the time-varying characteristic speed described in Remark~\ref{remark:congestion_wave} to track the disturbance. The wavefront generated by our model continuously adjusts to instantaneous vehicle speeds and spacings, as shown in Figure~\ref{fig:case1}\,(left), enabling it to more accurately follow the disturbance path. Figure~\ref{fig:case1}\,(right) confirms this behavior: speed deviations predicted by the proposed method cluster near zero, whereas those from the constant-speed model are spread more evenly across approximately \([-2,2]\)\,m/s.

% no-phase trans, single
\begin{figure}[!h]\centering
    \includegraphics[width=1\textwidth]{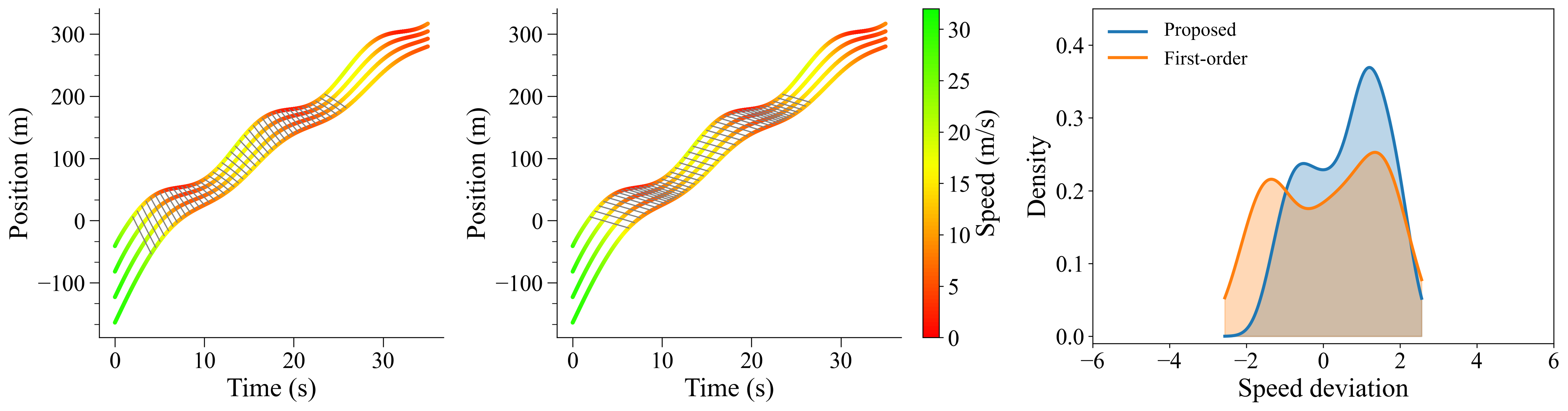}
    \vspace{-10pt} % Adjust vertical space between figures and caption
    \caption{Path of the proposed wave (left) and constant speed wave (middle) for a vehicular platoon under Case 1, and the probability density function of vehicle‐pair speed deviations (right)}
    \vspace{-10pt}
    \label{fig:case1}
\end{figure}

\subsubsection{Case 2: Single Oscillation with Initial disturbance}
In Case~2, an ACC cut-in occurs in front of the second vehicle at \(t = 10\ \mathrm{s}\), introducing an initial disturbance. The cut-in vehicle merges into the target lane and leaves a \(10\ \mathrm{m}\) gap to its leader. Prior to this event, all vehicles follow the ACC control law and travel under equilibrium conditions. Because the proposed wave speed follows a mass conservation law and does not account for lateral source terms, it is applicable only after the cut-in maneuver is complete. As the platoon reacts to the disturbance, it generates a brief positive wave similar to the recovery wave observed when a bottleneck clears. Figure~\ref{fig:case2}\,(left) and \,(middle) show that along the identified wave path, the vehicle-pair speed deviations obtained with the proposed method remain lower than those from the constant speed model, especially immediately after the initial disturbance, as highlighted in the zoomed view.

% no-phase trans, single, initial disturbance
\begin{figure}[!ht]\centering
    \includegraphics[width=1\textwidth]{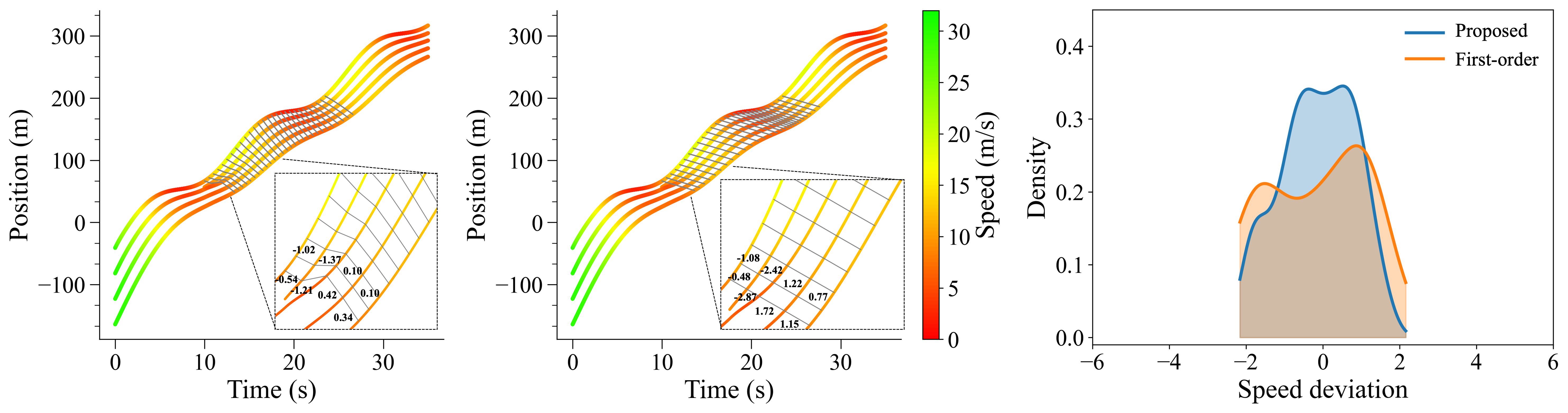}
    \caption{Path of the proposed wave (left) and constant speed wave (middle) for a vehicular platoon under Case 2, and the probability density function of vehicle‐pair speed deviations (right)}
    \label{fig:case2}
    \vspace{-10pt}
\end{figure}

\subsubsection{Case 3: Compound Steady Oscillation}
In the compound steady oscillation scenario, both the amplitude and frequency of the oscillations are higher than in the single oscillation case. Figure~\ref{fig:case3}\,(right) shows that speed deviations increase for both models, yet the proposed traffic wave still outperforms the constant speed benchmark. The maximum vehicle-pair speed deviation under our method is \(3.47\ \mathrm{m/s}\) (average \(1.27\ \mathrm{m/s}\)), whereas the constant speed model reaches \(5.81\ \mathrm{m/s}\) (average \(2.07\ \mathrm{m/s}\)). The larger deviations arise from the source term in the ACC-embedded balance law (Eq.~\eqref{ACC-traffic}), which captures departures from equilibrium spacing. Although the characteristic speeds, defined by the homogeneous flux Jacobian, remain the same as in the absence of a source term, the source continually perturbs vehicle accelerations, causing the state to evolve and the local characteristic speed \(\lambda_2(\rho,v)\) to change over time. Consequently, larger oscillations lead to larger speed deviations. Even so, the proposed wave model, with its time varying characteristic speed, maintains lower deviations than the constant speed model.

% no-phase trans, compound
\begin{figure}[!ht]\centering
    \includegraphics[width=1\textwidth]{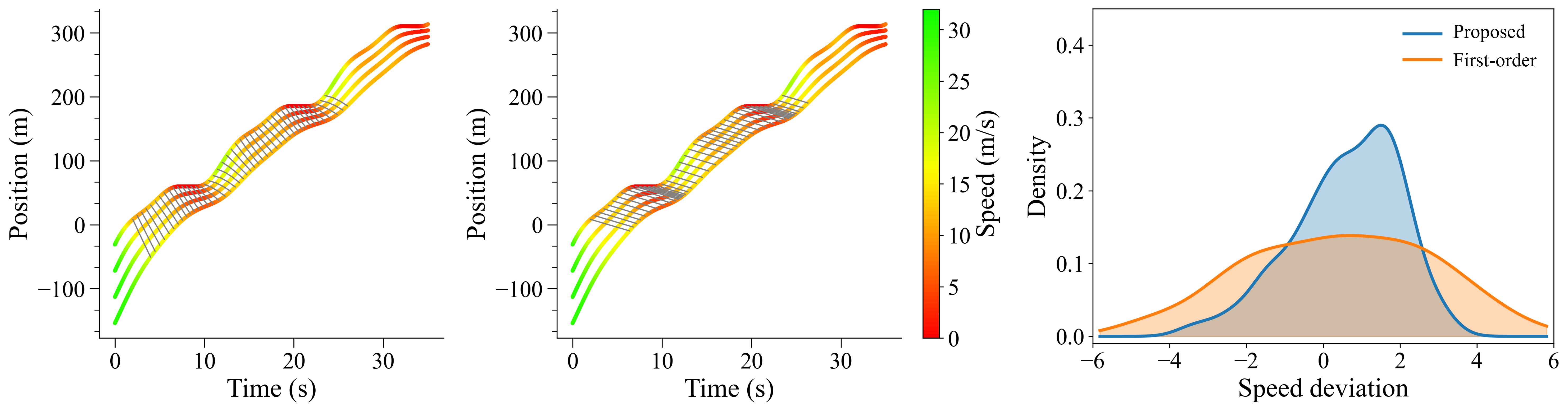}
    \caption{Path of the proposed wave (left) and constant speed wave (middle) for a vehicular platoon under Case 3, and the probability density function of vehicle‐pair speed deviations (right)}
    \label{fig:case3}
    \vspace{-10pt}
\end{figure}

\subsubsection{Case 4: Phase Transition}
Case 4 examines the transition from free flow to congestion, a regime where both characteristic and shock waves contribute to disturbance propagation. Vehicles travel at the free-flow speed \(v_f\) until their inter-vehicle spacing falls to the ACC activation threshold \(s_c = \tau v_f + L\). At that moment, the ACC controller engages, generating a spacing discontinuity that initiates a shock wave. The transition interval begins at the engagement time \(t^{\star}\) and ends once the spacing reaches \(s_e(v) = \tau v_e + L\). This is the steady-state spacing prescribed by the constant time-headway policy for a given equilibrium speed \(v_e\) from oscillation decomposition. During transition, the disturbance propagates with the shock speed described in Remark~\ref{shock}; outside the transition interval, propagation follows the characteristic speed \(\lambda_2(\rho, v)\). Figure~\ref{fig:case4}\,(left) shows both a shock wave and a characteristic wave. When the shock wave overtakes the characteristic wave, the latter stops propagating. Under these conditions the proposed model keeps vehicle-pair speed deviations tightly clustered near zero, with an average of \(0.89\ \text{m/s}\) compared with \(1.29\ \text{m/s}\) for the constant speed model, a reduction of about 31\,\%. This result confirms the reliability of the proposed traffic wave across diverse traffic conditions.

\begin{figure}[!ht]\centering
    \includegraphics[width=1\textwidth]{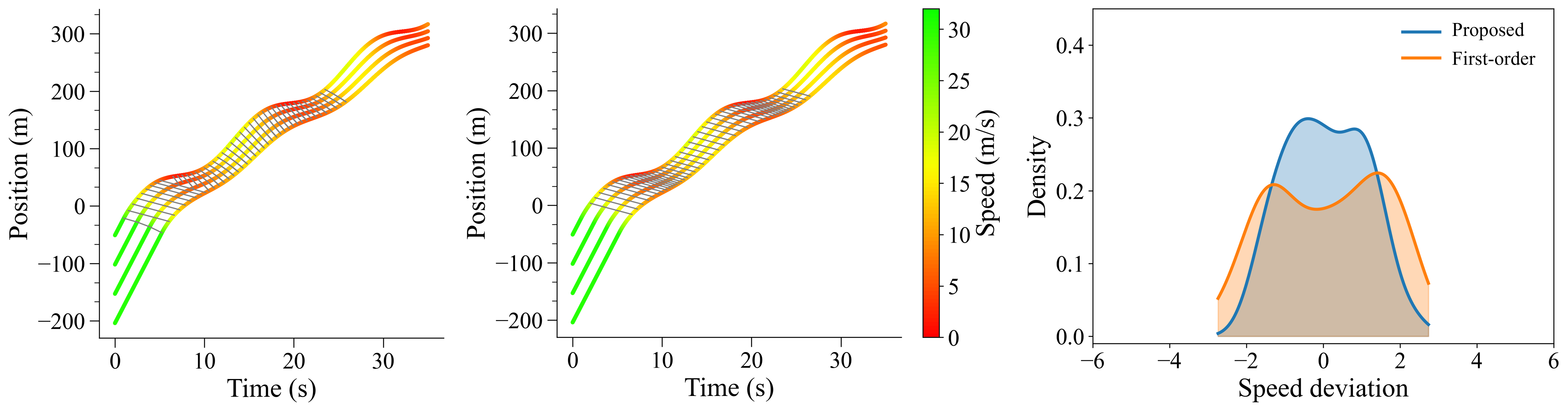}
    \caption{Path of the proposed wave (left) and constant speed wave (middle) for a vehicular platoon under Case 4, and the probability density function of vehicle‐pair speed deviations (right)}
    \label{fig:case4}
    \vspace{-10pt}
\end{figure}

\subsection{Empirical Validation}
For empirical validation, we use the OpenACC dataset~\citep{makridis2021openacc}, which captures realistic commercial ACC behavior. A trajectory with a pronounced speed dip, typical of oscillations in congested traffic, is selected (Figure~\ref{fig:openacc}~(a)). The exact control law used in the commercial ACC is not publicly known, whereas our framework is derived for a linear feedback controller. Therefore, we reproduce the car-following trajectory using calibrated control parameters. The ACC control parameters were calibrated with approximate Bayesian computation; see \citep{jiang2024generic, jiang2024dynamic} for details. We randomly sample from the resulting joint distribution of \((\tau, k_s, k_v, L)\) to generate realistic stochastic ACC trajectories and evaluate the proposed traffic wave accordingly.

Figure~\ref{fig:openacc} shows the car-following speed profile (a) alongside the wave paths generated by the proposed traffic wave (c) and the constant-speed wave (d) for a selected control parameter set. The statistics result in (c) and (d) correspond to this same parameter set. Because the platoon remains in the congested regime, we depict only the characteristic wave (Remark~\ref{remark:congestion_wave}). Even under realistic oscillations, the wave paths of the proposed traffic wave differ significantly from those of the constant-speed wave.

To avoid bias from single parameter set, we randomly draw 200 samples from the calibrated joint distribution and simulated stochastic ACC trajectories. The resulting distribution based on vehicle-pair speed deviation appears in Figure~\ref{fig:openacc}(b). Across these simulations, the speed deviations under the proposed formulation cluster tightly around zero. The mean, median, lower quartile, and upper quartile of the absolute speed deviation are \(0.35\), \(0.24\), \(0.11\), and \(0.45\;\text{m/s}\), respectively, versus \(0.51\), \(0.45\), \(0.26\), and \(0.70\;\text{m/s}\) for the constant-speed wave. These statistics further confirm that the proposed traffic wave captures disturbance propagation more accurately.

\begin{figure}[!ht]\centering
    \vspace{-5pt}
    \includegraphics[width=0.65\textwidth]{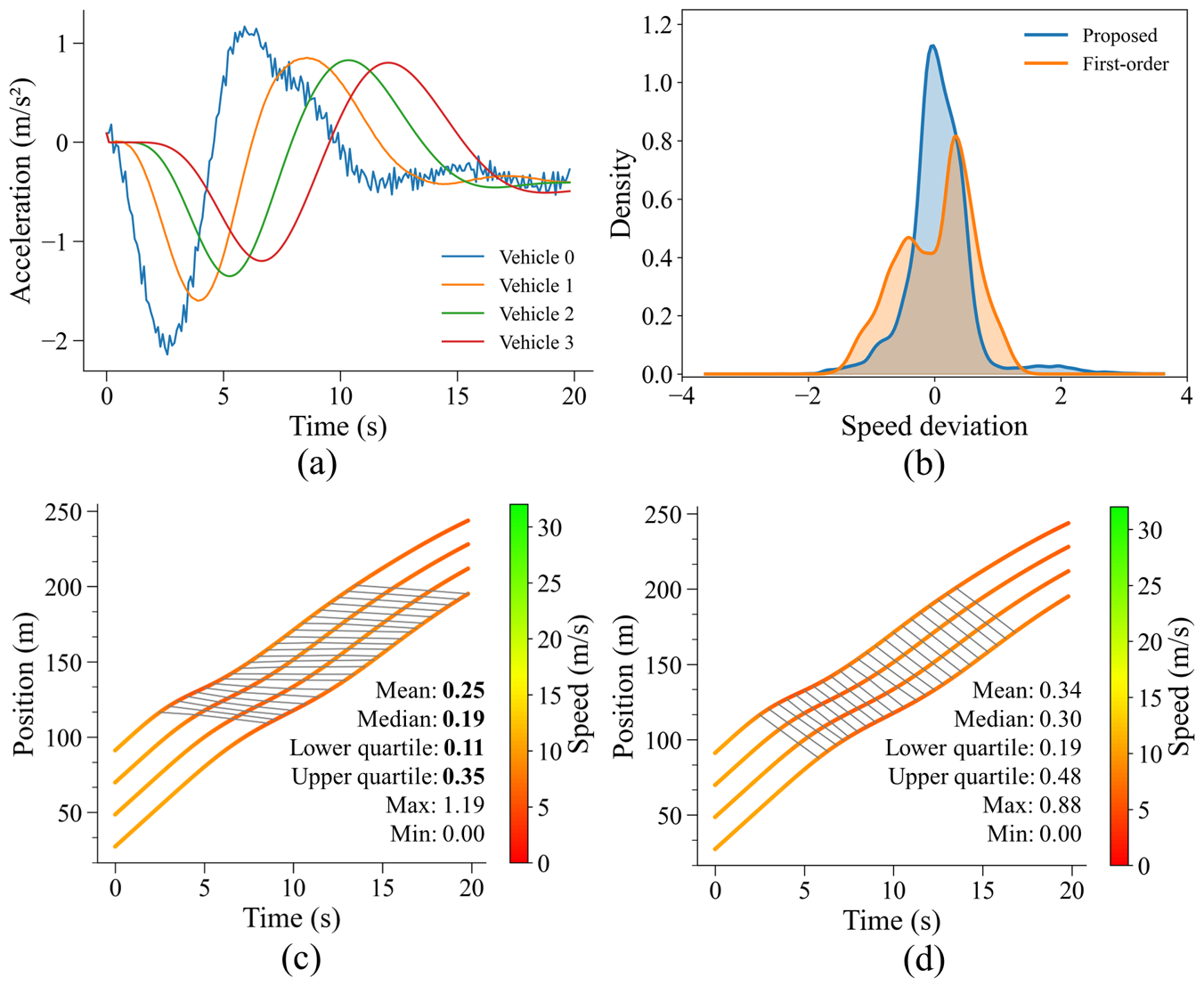}
    \vspace{-5pt}
    \caption{Disturbance propagation in OpenACC trajectories (control parameters for (a), (c), (d) are $\tau=1.0883$ s, $L=9.655$ m, $k_s=0.3134$, $k_v=0.4629$)}
    \label{fig:openacc}
    \vspace{-10pt}
\end{figure}

\section{Conclusion}
In this study, we introduce a macroscopic traffic flow model that embeds the ACC control law into the mass conservation framework. The resulting system constitutes a special case of the phase transition family of non-equilibrium models and, in the congested regime, is proved to be strictly hyperbolic and anisotropic, thereby avoiding the non-physical artifacts that hinder earlier higher-order formulations. Eigenanalysis of the flux Jacobian yields closed-form characteristic waves whose amplitudes depend jointly on the initial conditions and the ACC control gains. These waves govern disturbance propagation within the congested regime, whereas during free-flow–to-congestion transitions, the model predicts shock fronts whose trajectories are anchored by ACC engagement events. Once the spacing behind each vehicle reaches its equilibrium value, the dynamics revert to characteristic wave propagation.

Comprehensive numerical experiments, including steady single-frequency oscillations, compound-frequency oscillations, non-steady cut-in disturbances, and free flow to congestion phase transitions, corroborate the theory and show that the proposed model exhibits smaller vehicle-pair speed deviations along the traffic wave compared with a first-order LWR benchmark. A similar advantage is observed in the empirical OpenACC dataset. Because the proposed traffic wave evolves with local spacing and velocity, the model effectively captures disturbance propagation under diverse traffic conditions, including phase transitions from free flow to congestion, where both shock and characteristic waves play a role. By explicitly linking macroscopic wave behavior to microscopic controller parameters, the framework offers a tractable tool for evaluating the impact of ACC design on traffic congestion and propagation in an analytical and predictable manner, which could further inform traffic-management strategies that leverage automated vehicles' capabilities. Future work will extend the framework beyond pure, homogeneous ACC traffic to mixed traffic with human-driven vehicles and to other ACC settings and dynamics, and will further establish well-posedness and interface-aware numerical schemes for the complete phase-transition model.

\section*{Acknowledgment}
We would like to thank Dr. Shuo Yang and Dr. Liangze Yang for helpful discussions on the numerical scheme. We would like to acknowledge the anonymous reviewers for their valuable feedback, which helped improve the quality of the manuscript.

\appendix
\section{Comparison with Phase Transition Model}
\label{connection_PTM}
Phase transition traffic flow models were first developed by \cite{colombo20022,colombo2003hyperbolic}. These models were later extended by \cite{blandin2011general,blandin2013phase}, who adopted a fixed free-flow speed and incorporated multiple fundamental diagrams in the congested regime to capture the variability in density-flow relationships observed in empirical data. The extended fundamental diagram is expressed as:

\begin{equation}\label{eq:vptm}
v (u) =
\begin{cases}
v_f, & \text{in free-flow}~(\Omega_f) \\
v(\rho)(1 + p), & \text{in congestion}~(\Omega_c)
\end{cases}
\end{equation}
where, $u=(\rho, p)$ and $\rho$ is the traffic density. $p$ is a real-valued parameter, independent of $\rho$, that quantifies deviations from the equilibrium flow. $v(\rho)$ represents the equilibrium speed in congestion. 
% The deviation parameter $p$ is bounded by $p_{\min} \le p \le p_{\max}$, with $p_{\min} > -1$.
Therefore, the corresponding phase transition model is formulated as follows:

\begin{equation}\label{eq:phase-trans-pde}
\left\{
\begin{aligned}
&\rho_t + (\rho v)_x = 0 && \text{in free-flow}~(\Omega_f) \\[6pt]
&\left\{
\begin{aligned}
&\rho_t + (\rho v)_x = 0 \\
&p_t + (p v)_x = 0
\end{aligned}
\right. && \text{in congestion}~(\Omega_c)
\end{aligned}
\right.
\end{equation}
Its eigenvectors in the congested phase are

\begin{equation}
r_1(u)=
\begin{pmatrix}
\rho \\ p
\end{pmatrix},
\qquad
r_2(u)=
\begin{pmatrix}
v(\rho) \\ -(1+p)\, v'(\rho)
\end{pmatrix}.
\end{equation}
The corresponding eigenvalues are

\begin{equation}
\label{egv_tfm}
\lambda_1(u)=v(u)+(1 + p)V'(\rho)\rho + V(\rho)p, 
\qquad
\lambda_2(u)=v(u).
\end{equation}

The first characteristic field, $\lambda_1$, is genuinely nonlinear, whereas the second, $\lambda_2$, is linearly degenerate; for a detailed proof see \citep{blandin2013phase}. Using the eigenvalues of the ACC model (Eq.~\eqref{eigenvalue}) and those of the phase transition model (Eq.~\eqref{egv_tfm}), the two systems coincide when their non-trivial eigenvalues match, i.e.\ when $\lambda^{\mathrm{PTM}}_1 = \lambda^{\mathrm{ACC}}_2$, because both models already share the contact characteristic $\lambda = v$.

We introduce the bijection
\begin{equation}\label{eq:bijection}
C : (\rho, v) \mapsto (\rho, p), 
\qquad 
p = \frac{v}{V(\rho)} - 1,
\end{equation}
which maps an ACC state $(\rho, v)$ onto the extended fundamental diagram $v = V(\rho)(1 + p)$ (Eq.~\eqref{eq:vptm}). Substituting this relation into the phase transition eigenvalue and equating it to $\lambda^{\mathrm{ACC}}_2$ gives the feedback gain
\begin{equation}\label{kv_ptm}
k_v = -\rho\Bigl[v + \frac{\rho V'(\rho)}{V(\rho)}\,v - V(\rho)\Bigr].
\end{equation}

With this choice of $k_v$, the two eigenstructures are identical. In practical ACC implementations, however, the control gain $k_v$ is typically taken as a constant in the congested regime, whereas Eq.~\eqref{kv_ptm} prescribes a gain that is an explicit function of both $\rho$ and $v$.

\section{Comparison with LWR+ARZ Phase Transition Model}
\label{connection_LWR_ARZ}
Several phase transition traffic flow models have combined the LWR model for free-flow traffic with the second-order ARZ model for congested regimes~\citep{goatin2006aw,benyahia2016entropy}. We further examine the connection and differences between our proposed model and these LWR with ARZ formulations, particularly in the congested regime.

Comparing the ACC momentum equation in Eq.~\eqref{law-v} with the ARZ momentum equation in Eq.~\eqref{ARZ_v} reveals a clear correspondence: if one sets the spacing‐feedback gain \(k_s=0\) and identifies the “sound‐speed” \(c(\rho)=-\,k_v/\rho\), then Eq.~\eqref{law-v} reduces identically to Eq.~\eqref{ARZ_v}.  In other words, the ARZ momentum equation can be viewed as the special case of our ACC momentum equation when spacing-feedback gain is removed. Despite this structural similarity, the two models differ fundamentally in their characteristic fields.  In the ARZ model, one characteristic field remains linearly degenerate while the other is genuinely nonlinear, admitting both shocks and rarefactions \citep{zhang2002non}.  If the control gains in ACC-embedded is a function of density, the second characteristic field becomes genuinely nonlinear while the first remains degenerate. For example in Eq.~\eqref{eq:linear-degen}, let \(k_v = k_v(\rho)\). Then
\begin{equation}
\lambda_2(\rho,v) = v - \frac{k_v(\rho)}{\rho}, 
\quad
r_2 = \begin{pmatrix}1 \\[3pt] -\dfrac{k_v(\rho)}{\rho^2}\end{pmatrix},
\end{equation}

In that case, \(\nabla\lambda_2 = \bigl(-k_v'(\rho)/\rho + k_v(\rho)/\rho^2,\;1\bigr)^\mathsf{T}\), so \(\nabla\lambda_2\cdot r_2 = -\,k_v'(\rho)/\rho \neq 0\) whenever \(k_v\) depends on \(\rho\).

\section{Numerical Scheme for the Congested Regime of the ACC-embedded Traffic Flow Model}
\label{app:numerical_scheme}

This appendix summarizes the numerical scheme used for the ACC-embedded traffic flow model in the congested regime $\Omega_c$ (ACC active). We consider smooth solutions and use a first-order finite-volume (Godunov-type) discretization on a uniform grid~\citep{toro2013riemann}. Let $(\rho_i^n,v_i^n)$ denote cell-average approximations of density $\rho(x,t)$ and speed $v(x,t)$ at time $t^n$, with spatial and temporal steps $\Delta x$ and $\Delta t$.

In $\Omega_c$, the governing system is
\begin{equation}
\left\{
\begin{array}{l}
\rho_t + (\rho v)_x = 0,\\[2pt]
v_t + \Bigl(v-\dfrac{k_v}{\rho}\Bigr) v_x = k_s\Bigl(\dfrac{1}{\rho}-\tau v-L\Bigr),
\end{array}
\right.
\label{eq:app_system}
\end{equation}
where $\tau$ is the desired time headway, $L$ is the jam-spacing parameter, and $k_s$, $k_v$ are ACC gains.

For the mass conservation equation, we use the Rusanov (local Lax--Friedrichs) numerical flux~\citep{toro2013riemann}
\begin{equation}
F_{i+\frac12}^n=
\frac12\Bigl(F(\rho_i^n,v_i^n)+F(\rho_{i+1}^n,v_{i+1}^n)\Bigr)
-\frac12\,\alpha_{i+\frac12}^n\bigl(\rho_{i+1}^n-\rho_i^n\bigr),
\label{eq:app_rusanov}
\end{equation}
where $F(\rho,v)=\rho v$ is the physical mass flux. The interface dissipation speed $\alpha_{i+\frac12}^n$ is chosen as a local wave-speed bound,
\begin{equation}
\alpha_{i+\frac12}^n=
\max\Bigl(
|\lambda_1(\rho_i^n,v_i^n)|,\,
|\lambda_2(\rho_i^n,v_i^n)|,\,
|\lambda_1(\rho_{i+1}^n,v_{i+1}^n)|,\,
|\lambda_2(\rho_{i+1}^n,v_{i+1}^n)|
\Bigr),
\label{eq:app_alpha}
\end{equation}
with characteristic speeds
\begin{equation}
\lambda_1(\rho,v)=v,
\qquad
\lambda_2(\rho,v)=v-\frac{k_v}{\rho}.
\label{eq:app_lambdas}
\end{equation}

For the velocity equation written in nonconservative transport form, the convective term is an advection operator for $v$ with coefficient multiplying $v_x$. We therefore define the advection speed
\begin{equation}
a(\rho,v)=v-\frac{k_v}{\rho},
\label{eq:app_a}
\end{equation}
and discretize the convective part using a first-order upwind flux: with $a_i^n=a(\rho_i^n,v_i^n)$ and interface speed $a_{i+\frac12}^n=\tfrac12(a_i^n+a_{i+1}^n)$, the interface state of $v$ is taken from the left cell if $a_{i+\frac12}^n\ge 0$ and from the right cell otherwise.

The source term is incorporated via operator splitting~\citep{leveque2002finite}: after the homogeneous (flux) update, we advance the source ordinary differential equation (ODE) for $v$ cellwise using an explicit step with the updated density,
\begin{equation}
v_i^{n+1}
=
v_i^{\star}
+\Delta t\,k_s\Bigl(\frac{1}{\rho_i^{n+1}}-\tau v_i^{\star}-L\Bigr),
\label{eq:app_source}
\end{equation}
where $v_i^{\star}$ denotes the velocity after the convective update and before applying the source term. The time step is chosen by a CFL condition based on the same wave-speed bound,
\begin{equation}
\Delta t
=
\mathrm{CFL}\,
\frac{\Delta x}{\max_i\max\bigl(|\lambda_1(\rho_i^n,v_i^n)|,\;|\lambda_2(\rho_i^n,v_i^n)|\bigr)},
\label{eq:app_cfl}
\end{equation}
with a prescribed $\mathrm{CFL}\in(0,1)$.

This appendix describes the numerical scheme only for the congested regime $\Omega_c$ (ACC active), where the dynamics form a hyperbolic balance law with a nonzero source term. A complete scheme for the full ACC-embedded model is more complicated, since it also needs to track the moving cruise-to-control switching interface and provide a consistent coupling rule across it, i.e., specify how the left/right states and the associated fluxes connect when the solution transitions between $\Omega_f$ and $\Omega_c$.

\section{Validation of the Congested-Regime Numerical Scheme Against Microscopic Simulations}
\label{app:pde_micro_validation}

This appendix reports a direct numerical comparison between (i) a microscopic ACC car-following simulation on a ring and (ii) the macroscopic congested-regime balance law in \(\Omega_c\) solved by the numerical scheme in Appendix~\ref{app:numerical_scheme}. The goal is to verify that, under matched settings and within \(\Omega_c\) (ACC active), the macroscopic PDE reproduces the Eulerian fields implied by the microscopic dynamics.

We use a ring-road (periodic) setting so that the microscopic ACC simulation and the macroscopic PDE share the same periodic boundary condition, and the comparison is not influenced by inflow or outflow boundary treatments~\citep{orosz2010traffic,kesting2013traffic,ntousakis2015microscopic}. The spatial domain is \([0,L_x)\), with periodic identification at \(x=0\) and \(x=L_x\). The ring length \(L_x\) is determined by the initial microscopic configuration: we initialize vehicles on the ACC time-headway manifold \(s_i(0)=\tau v_i(0)+L_0\) and set
\begin{equation}
L_x=\sum_{i=1}^{N} s_i(0),
\label{eq:appD_Lx}
\end{equation}
so that the PDE and the microscopic model evolve on the same closed loop with consistent total spacing.

For the microscopic simulation, \(N\) vehicles evolve on the ring with wrap-around leader identification. At each microscopic time step, vehicles are ordered by position; each vehicle uses its immediate predecessor on the ring as the leader; the gap is computed with periodic wrap-around; and the linear ACC control law in \(\Omega_c\) is applied using the desired time-headway spacing \(s^*(v)=\tau v+L_0\). The state is advanced with an explicit time step. To compare with the Eulerian PDE, the microscopic state is converted to Eulerian fields on a uniform grid \(x_m\in[0,L_x)\). Each vehicle \(i\) is associated with a road segment of length \(s_i(t)\) (with periodic wrap-around), and the micro-derived Eulerian fields are defined piecewise by
\begin{equation}
\rho_{\text{micro}}(x,t)=\frac{1}{s_i(t)},
\qquad
v_{\text{micro}}(x,t)=v_i(t),
\qquad
x \in \text{segment of vehicle } i .
\label{eq:appD_micro_to_eulerian}
\end{equation}
This produces \(\rho_{\text{micro}}(t,\cdot)\) and \(v_{\text{micro}}(t,\cdot)\) at the time $t$ used for the PDE output.

For the macroscopic simulation, we solve the congested-regime balance law on the same ring \([0,L_x)\) with periodic boundary conditions, using the finite-volume discretization in Appendix~\ref{app:numerical_scheme}. The PDE initial condition is taken directly from the micro-derived Eulerian fields at \(t=0\),
\begin{equation}
\rho(x,0)=\rho_{\text{micro}}(x,0),
\qquad
v(x,0)=v_{\text{micro}}(x,0),
\qquad
x\in[0,L_x),
\label{eq:appD_pde_ic}
\end{equation}
so that both models start from the same Eulerian state on the same domain. The PDE is advanced with a CFL time step; the mass equation uses the Rusanov flux, the convective part of the velocity equation uses an upwind discretization based on \(a(\rho,v)=v-k_v/\rho\), and the source term is applied by operator splitting as described in Appendix~\ref{app:numerical_scheme}. To quantify agreement over the full temporal horizon, we compute a space--time root-mean-square error (RMSE) by averaging the squared discrepancy across the entire Eulerian field $(t,x)$ sampled on the time grid $\{t^n\}_{n=1}^{N_t}$ and the spatial grid $\{x_m\}_{m=1}^{n_x}$,

\begin{equation}\label{eq:appD_rmse}
\mathrm{RMSE}_{\psi}
=
\sqrt{\frac{1}{N_t\,n_x}\sum_{n=1}^{N_t}\sum_{m=1}^{n_x}
\Bigl(\psi_{\mathrm{PDE}}(t^n,x_m)-\psi_{\mathrm{micro}}(t^n,x_m)\Bigr)^2},
\qquad \psi \in \{v,\rho\}.
\end{equation}

The comparisons for Cases~1--3 are presented in Figures~\ref{fig:case1_pde_speed_density_vertical}--\ref{fig:case3_pde_speed_density_vertical}, all of which remain in \(\Omega_c\). Across these cases, the Eulerian speed and density heatmaps produced by the proposed PDE closely match the microscopic benchmark in both the spatiotemporal structure and the propagation of wave patterns. Quantitatively, the space--time RMSE remains small, with \(\mathrm{RMSE}_{v}\) ranging from \(0.45\) to \(0.71\,\mathrm{m/s}\) and \(\mathrm{RMSE}_{\rho}\approx 0.002\,\mathrm{m^{-1}}\). Case~4 is not included because it involves the cruise-to-control phase transition and would require an interface-aware coupling treatment across \(\Omega_f\) and \(\Omega_c\).

\begin{figure}[htbp!]
  \centering
  \setlength{\abovecaptionskip}{2pt}
  \setlength{\belowcaptionskip}{0pt}
  \begin{subfigure}[t]{1\linewidth}
    \centering
    \includegraphics[width=0.85\linewidth]{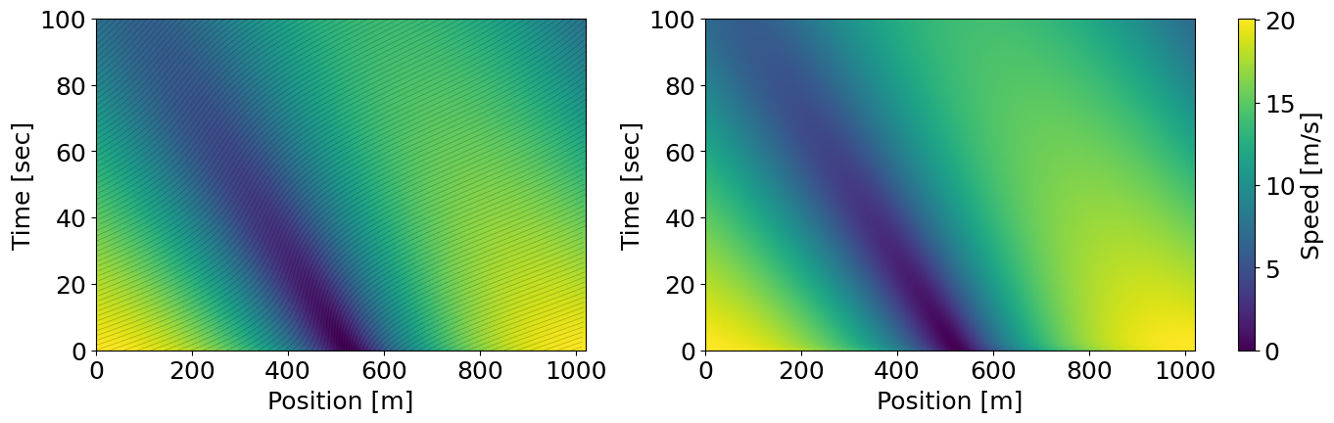}
    \label{fig:case1_speed_pde}
  \end{subfigure}
  \vspace{-12pt}
  \begin{subfigure}[t]{1\linewidth}
    \centering
    \includegraphics[width=0.85\linewidth]{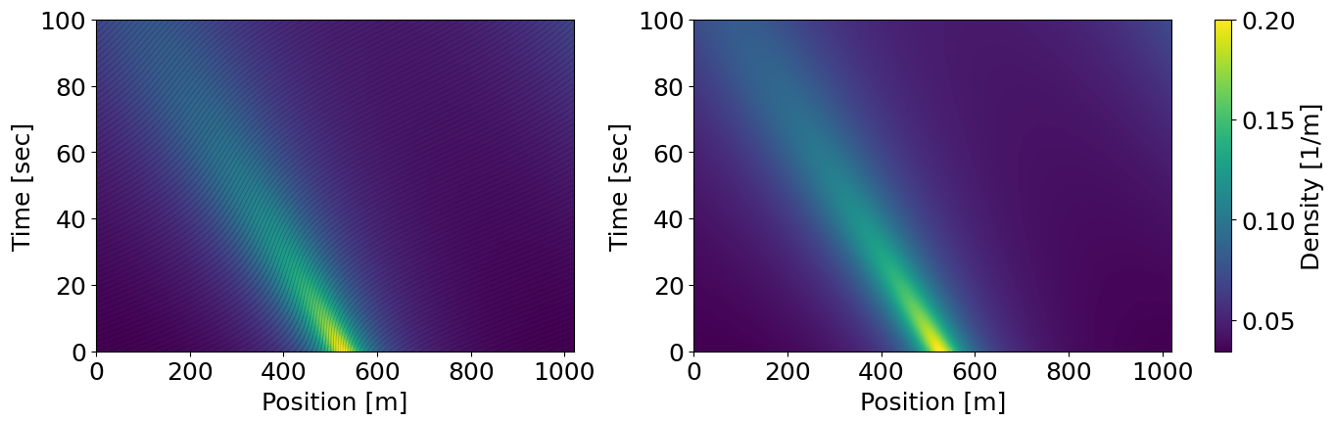}
    \label{fig:case1_density_pde}
  \end{subfigure}
  \vspace{-2pt}
  \caption{Speed and density fields of Case 1: microscopic trajectories mapped to Eulerian coordinates (left) and proposed PDE solution (right) ($\mathrm{RMSE}_{v}=0.45\,\mathrm{m/s}$, $\mathrm{RMSE}_{\rho}=0.002\,\mathrm{m^{-1}}$).}
  \label{fig:case1_pde_speed_density_vertical}
\end{figure}

\begin{figure}[htbp!]
  \centering
  \setlength{\abovecaptionskip}{2pt}
  \setlength{\belowcaptionskip}{0pt}
  \begin{subfigure}[t]{\linewidth}
    \centering
    \includegraphics[width=0.8\linewidth]{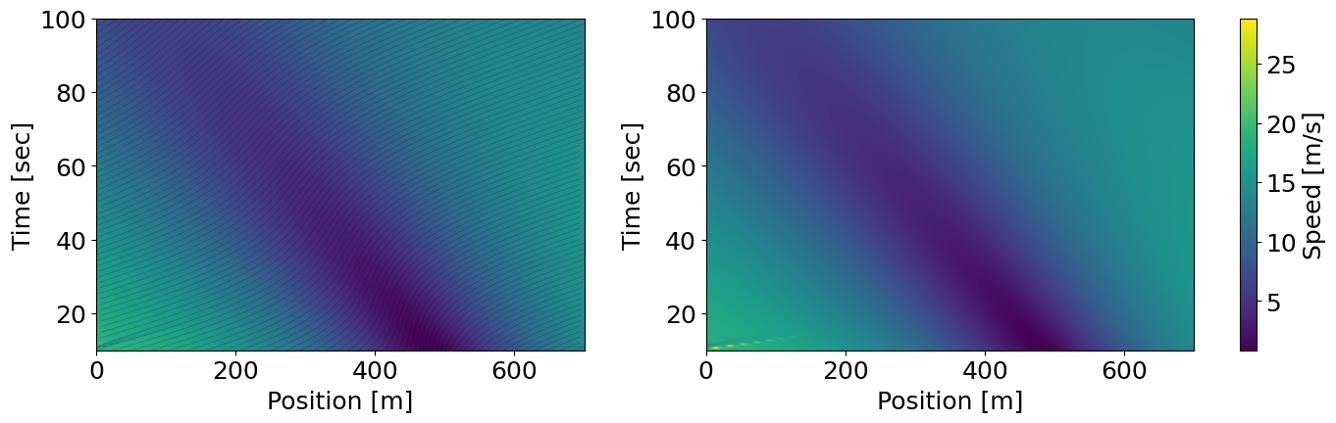}
    \label{fig:case2_speed_pde}
  \end{subfigure}
  \vspace{-12pt}
  \begin{subfigure}[t]{\linewidth}
    \centering
    \includegraphics[width=0.8\linewidth]{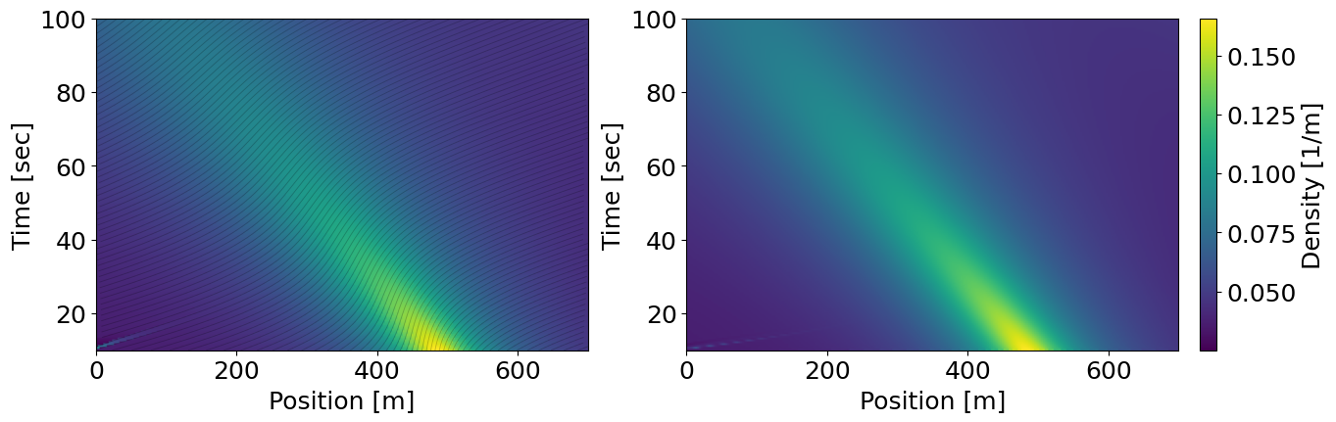}
    \label{fig:case2_density_pde}
  \end{subfigure}
  \vspace{-2pt}
  \caption{Speed and density fields of Case 2: microscopic trajectories mapped to Eulerian coordinates (left) and proposed PDE solution (right) ($\mathrm{RMSE}_{v}=0.50\,\mathrm{m/s}$, $\mathrm{RMSE}_{\rho}=0.002\,\mathrm{m^{-1}}$).}
  \label{fig:case2_pde_speed_density_vertical}
\end{figure}

\begin{figure}[htbp!]
  \centering
  \setlength{\abovecaptionskip}{2pt}
  \setlength{\belowcaptionskip}{0pt}
  \begin{subfigure}[t]{1\linewidth}
    \centering
    \includegraphics[width=0.8\linewidth]{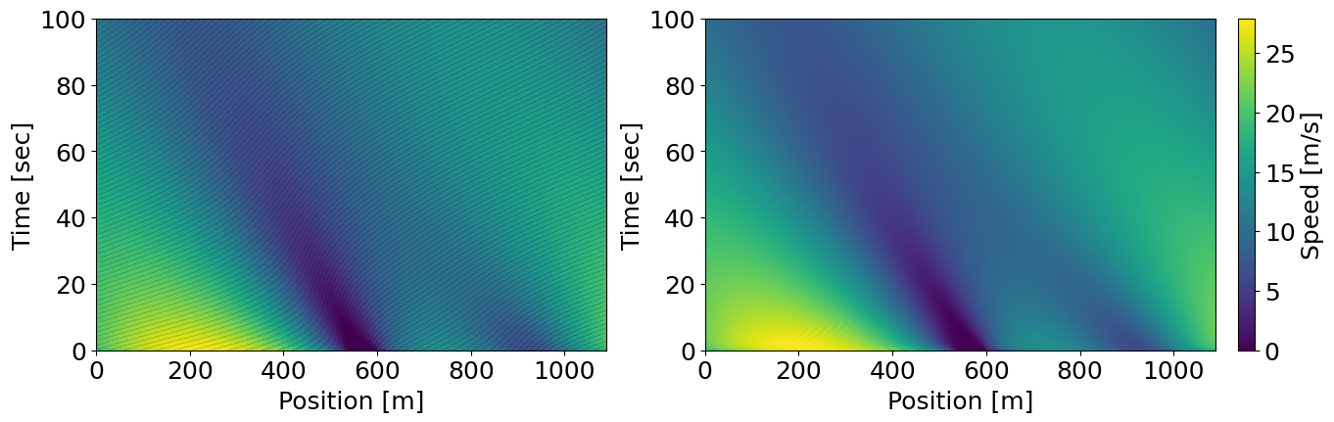}
    \label{fig:case3_speed_pde}
  \end{subfigure}
  \vspace{-12pt}
  \begin{subfigure}[t]{1\linewidth}
    \centering
    \includegraphics[width=0.8\linewidth]{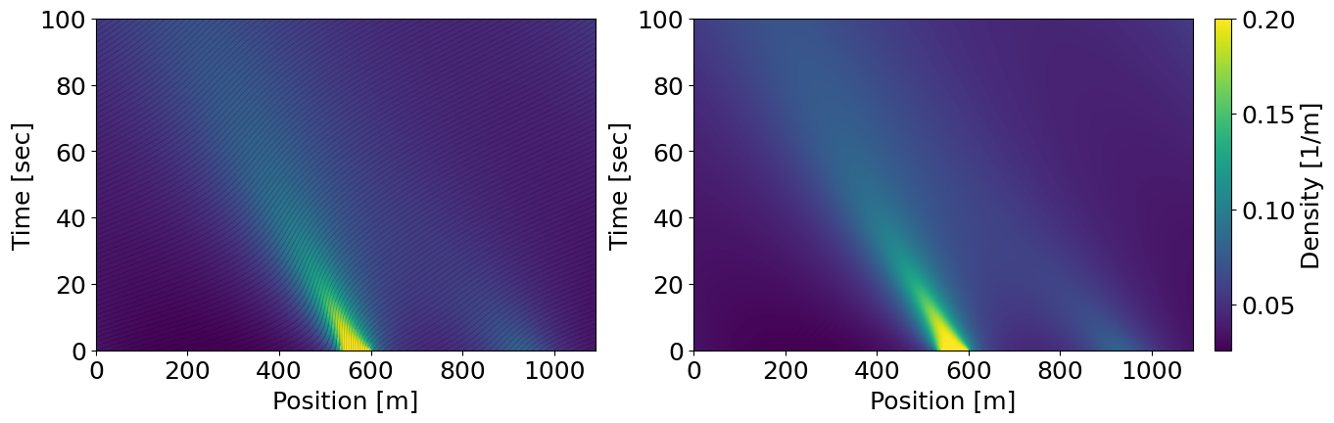}
    \label{fig:case3_density_pde}
  \end{subfigure}
  \vspace{-2pt}
  \caption{Speed and density fields of Case 3: microscopic trajectories mapped to Eulerian coordinates (left) and proposed PDE solution (right) ($\mathrm{RMSE}_{v}=0.71\,\mathrm{m/s}$, $\mathrm{RMSE}_{\rho}=0.002\,\mathrm{m^{-1}}$).}
  \label{fig:case3_pde_speed_density_vertical}
\end{figure}

Because an OpenACC speed profile is not inherently periodic, directly repeating it would create a discontinuity at the period boundary. For the ring-road setting, we therefore use an Fast Fourier Transform (FFT) based filtering to construct a smooth periodic approximation. The resulting Fourier reconstruction enforces periodicity while preserving the overall trend and dominant oscillations, and it is used to initialize the benchmark disturbance (Figure~\ref{fig:ring_empirical_fft_top10}). Using this OpenACC-derived periodic initialization, the proposed PDE continues to match the microscopic benchmark in both the speed and density fields. The resulting heatmaps remain highly consistent in wave structure and evolution, and the overall discrepancy is small, with $\mathrm{RMSE}_{v}=0.11\,\mathrm{m/s}$ and $\mathrm{RMSE}_{\rho}=0.003\,\mathrm{m^{-1}}$ (see Figure.~\ref{fig:empirical100_pde_speed_density_vertical}).

\begin{figure}[htbp]
  \centering
  \setlength{\abovecaptionskip}{2pt}
  \setlength{\belowcaptionskip}{0pt}
  \includegraphics[width=0.9\linewidth]{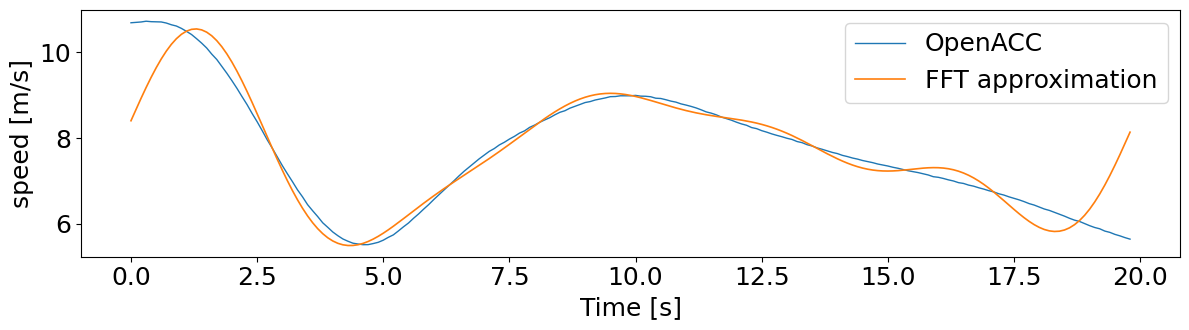}
  \vspace{-4pt}
  \caption{Ring-road empirical dataset: speed profile and Fourier reconstruction (RMSE = 0.60\,m/s).}
  \label{fig:ring_empirical_fft_top10}
  \vspace{-6pt}
\end{figure}

\begin{figure}[htbp]
  \centering

  \begin{subfigure}[t]{\linewidth}
    \centering
    \includegraphics[width=0.8\linewidth]{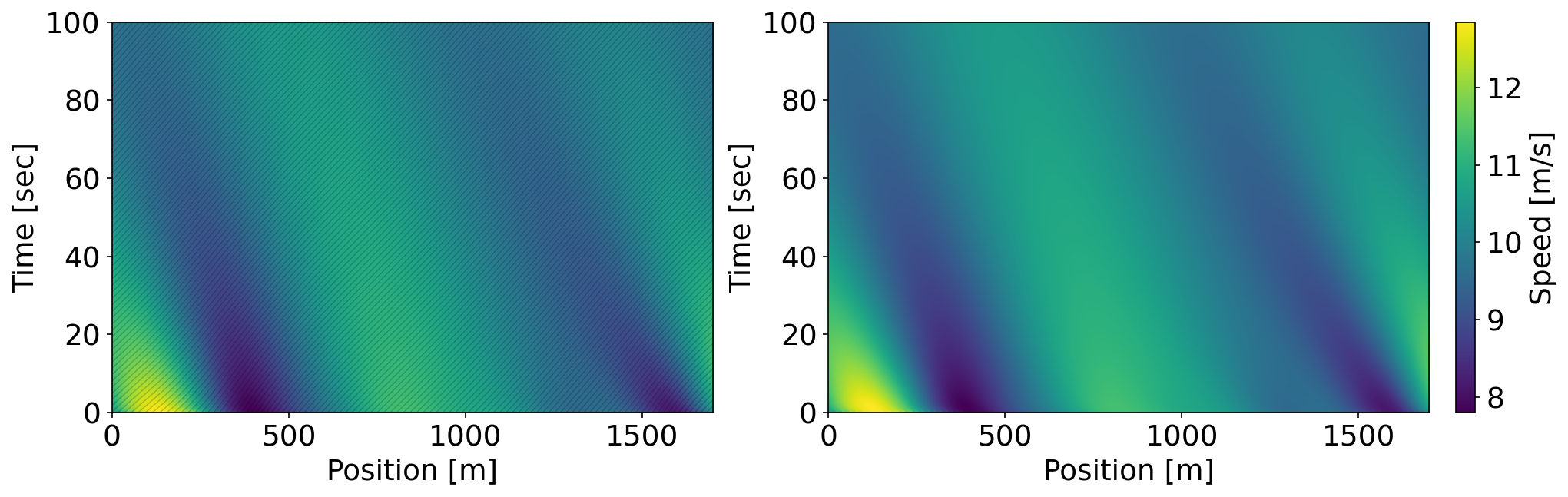}
    \label{fig:openacc_speed_pde}
  \end{subfigure}
  \vspace{-12pt}
  \begin{subfigure}[t]{\linewidth}
    \centering
    \includegraphics[width=0.8\linewidth]{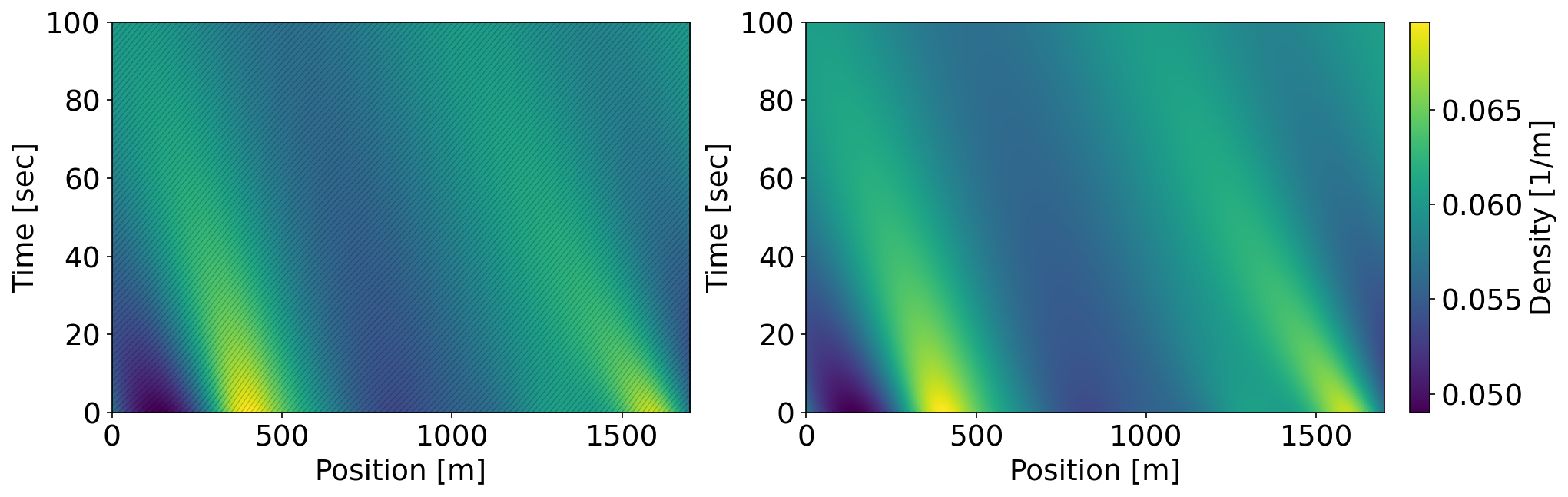}
    \label{fig:openacc_density_pde}
  \end{subfigure}
  \caption{Speed and density fields of the OpenACC dataset: microscopic trajectories mapped to Eulerian coordinates (left) and proposed PDE solution (right) ($\mathrm{RMSE}_{v}=0.11\,\mathrm{m/s}$, $\mathrm{RMSE}_{\rho}=0.003\,\mathrm{m^{-1}}$).}
  \label{fig:empirical100_pde_speed_density_vertical}
\end{figure}

\bibliographystyle{apalike} 
\bibliography{reference}

\end{document}